\documentclass[reqno]{article}
\usepackage{amsmath,amsfonts,amssymb,bbm,epsfig}
\usepackage{xcolor,amsthm}
\usepackage{amssymb,amsthm,hyperref}
\usepackage{scalerel,stackengine}
\usepackage{authblk}

\newenvironment{clmproof}[1]{\begin{proof}[Proof of Claim~\ref{#1}]\let\qednow\qedsymbol\renewcommand{\qedsymbol}{}}{\; \qednow \end{proof}}

\newtheorem{theorem}{Theorem}[section]
\newtheorem{lemma}[theorem]{Lemma}

\newtheorem{claim}[theorem]{Claim}
\newtheorem{corollary}[theorem]{Corollary}
\newtheorem{obs}[theorem]{Observation}
\newtheorem{conjecture}[theorem]{Conjecture}

\theoremstyle{definition}
\newtheorem{definition}[theorem]{Definition}
\newtheorem{remark}[theorem]{Remark}

\newcommand{\newoddpage}{\newpage\ifodd\value{page}\else\ \newpage\fi}

\numberwithin{equation}{section}

\newcommand\pig[1]{\scalerel*[5pt]{\big#1}{%
  \ensurestackMath{\addstackgap[1.5pt]{\big#1}}}}
\newcommand\pigl[1]{\mathopen{\pig{#1}}}
\newcommand\pigr[1]{\mathclose{\pig{#1}}}

\renewcommand{\Pr}{\mathbb{P}}
\newcommand{\Ex}{\mathbb{E}}
\renewcommand{\le}{\leqslant}
\renewcommand{\ge}{\geqslant}

\newcommand{\cC}{\mathcal{C}}

\newcommand{\cF}{\mathcal{F}}
\newcommand{\cG}{\mathcal{G}}
\newcommand{\cH}{\mathcal{H}}

\newcommand{\cP}{\mathcal{P}}

\newcommand{\cS}{\mathcal{S}}
\newcommand{\cT}{\mathcal{T}}

\newcommand{\cW}{\mathcal{W}}

\newcommand\E{\operatorname{\mathbb E{}}}

\newcommand{\eps}{\varepsilon}
\newcommand{\cc}{\mathrm{c}}

\newcommand{\bb}[1]{\bigl(#1\bigr)}

\newcommand{\Bb}[1]{\Bigl(#1\Bigr)}

\newcommand{\Bigmid}{\Bigm|}

\newcommand{\tH}{\tilde{H}}

\newcommand{\hb}{{\hat{\theta}}}
\newcommand{\hC}{{\widehat{C}}}
\newcommand{\hCL}{{\hC_{\mathrm{L}}}}

\newcommand{\N}{\mathbb{N}}
\newcommand{\R}{\mathbb{R}}
\newcommand{\1}{\mathbbm{1}}
\newcommand{\ind}[1]{\1[#1]}

\newcommand{\sjn}{\sum_{j\textup{ neutral}}}
\newcommand{\pjn}{\prod_{j\text{ neutral}}}
\newcommand{\sjs}{\sum_{j\textup{ simple}}}
\newcommand{\pjs}{\prod_{j\text{ simple}}}
\newcommand{\sjc}{\sum_{j\text{ complex}}}
\newcommand{\pjc}{\prod_{j\text{ complex}}}

\newcommand\heps{\hat{\eps}}
\newcommand\hdelta{\hat{\delta}}

\newcommand{\Bin}{\mathrm{Bin}}

\newcommand{\floor}[1]{\lfloor #1 \rfloor}

\DeclareMathOperator{\Var}{Var}

\begin{document}

\title{Random cliques in random graphs revisited}


\author{Robert Morris\thanks{IMPA, Estrada Dona Castorina 110, Jardim Bot\^anico, Rio de Janeiro 22460-320, Brazil. Email: rob@impa.br} \thanks{RM was partially supported by FAPERJ (Proc.~E-26/200.977/2021) and by CNPq (Procs~303681/2020-9 and~407970/2023-1)} \hspace{0.01cm} and Oliver Riordan\thanks{Mathematical Institute, University of Oxford, Radcliffe Observatory Quarter, Woodstock Road, Oxford, OX2 6GG, UK. Email: riordan@maths.ox.ac.uk} \thanks{For the purpose of open access, the author has applied a CC BY public copyright licence to any author accepted manuscript arising from this submission.}}


\maketitle

\begin{abstract}
We study the distribution of the set of copies of some given graph $H$ in the random graph $G(n,p)$, focusing on the case when $H = K_r$. Our main results capture the `leading term' in the difference between this distribution and the `independent hypergraph model', where (in the case $H = K_r$) each copy is present independently with probability $\pi = p^{\binom{r}{2}}$. As a concrete application, we derive a new upper bound on the number of $K_r$-factors in $G(n,p)$ above the threshold for such factors to appear. 

We will prove our main results in a much more general setting, so that they also apply to random hypergraphs, and also (for example) to the case when $p$ is constant and $r = r(n) \sim 2\log_{1/p}(n)$. 
\end{abstract}

\section{Introduction}

Our aim in this paper is to study the distribution of the copies of some given graph $H$ in a binomial random graph $G(n,p)$. Intuitively, if we do not expect to see any pairs of copies of $H$ sharing edges, then not only will the number of copies of $H$ be asymptotically Poisson, but the random set of copies will be close to a random subset of all possible copies of $H$ in which each appears independently with probability $\pi  = p^{e(H)}$. As a warm-up for our main result, we formalise this in Section~\ref{sec_wu}. We then consider the case in which intersecting pairs are likely to appear, but larger `clusters' are unlikely, or there are not too many of them. Here we prove a result capturing the effect of intersecting pairs on the distribution of the set of copies, with an error term involving the expectation of larger clusters. Finally, we will give an application of this result, bounding the number of $K_r$-factors in $G(n,p)$ for $r\ge 3$ constant.

Throughout, we consider the random \emph{set} of copies of $H$, not just the number of such copies. This paper is in some sense inspired by~\cite{R-copies}, where the second author studied copies of $K_r$ in $G(n,p)$ via a coupling argument, but the technical approach we take here will be very different, arguing not via coupling, but by directly estimating the probability of each `plausible' outcome.

We will state and prove our main results in a more general setting (so that they also apply, for example, to hypergraphs), which makes them a little hard to absorb quickly (see Theorems~\ref{thmain} and~\ref{th2way}). We will therefore begin by giving two relatively simple consequences for the case of copies of $K_r$ in $G(n,p)$, first for cliques of constant size, and then for larger cliques when $p$ is constant.

\subsection{Constant-size cliques in $G(n,p)$}\label{sscl1}

For any graph $G$ and $r \in \N$, let us write $H_r(G)$ for the $r$-uniform hypergraph on $V(G)$ that encodes the $r$-cliques in $G$, that is, 
$$S \in E(H_r(G)) \qquad \Leftrightarrow \qquad G[S] \cong K_r.$$ 
Our first theorem bounds the probability that $H_r(G(n,p))$ is equal to a fixed hypergraph $H$ (for `most' choices of $H$) 
as long as $\mu_r \le n^{1 + \eps}$ for some small constant $\eps > 0$, where
$$\mu_r = {n \choose r} p^{{r \choose 2}}$$ 
denotes the expected number of edges of $H_r(G(n,p))$. In order to state it, let
\begin{equation}\label{def:Delta}
\Lambda(n,r) = \frac{1}{2} \sum_{s = 2}^{r-1} \binom{n}{r}\binom{r}{s}\binom{n-r}{r-s} p^{2\binom{r}{2}} \Big( p^{-\binom{s}{2}} - 1 \Big)
\end{equation}
and observe that $\Lambda(n,r)$ is the difference between the expected number of pairs of edge-overlapping copies of $K_r$ in $G(n,p)$, and the expected number of pairs of hyperedges in the random $r$-uniform hypergraph $H_r(n,\pi)$ that intersect in at least two vertices, where $\pi = p^{\binom{r}{2}}$. Moreover, given an $r$-uniform hypergraph $H$, let us write $G(H)$ for the graph whose edges are all pairs of vertices of $H$ that are contained in some hyperedge of $H$, and define
\begin{equation}\label{tdef}
 t(H) = \binom{r}{2} e(H) - e\big( G(H) \big).
\end{equation}
Observe that $t(H)$ is equal to the number of `repeated' pairs of vertices of $H$, where a pair that is contained in $k$ hyperedges is counted $k-1$ times. 

We are now ready to state the first application of our main theorems.

\begin{theorem}\label{thm:Kr:intro}
For each $r \ge 5$, there exists $\gamma > 0$ such that if\/ $p \le n^{-2/r+\gamma}$ and~$\mu_r \to \infty$, then the following holds. There exists a collection $\cG$ of $r$-uniform hypergraphs, with $\Pr\big( H_r(G(n,p)) \in \cG \big) \to 1$ as $n \to \infty$, such that\,\footnote{Here, and throughout the paper, $\omega = \omega(n)$ is a function with $\omega(n) \to \infty$ as $n \to \infty$ arbitrarily slowly. The rate at which $\Pr\big( H_r(G(n,p)) \in \cG \big) \to 1$ in Theorem~\ref{thm:Kr:intro} depends on $\omega$.}
$$\Pr\bb{ H_r(G(n,p)) = H } = \pi^{e(H)} (1-\pi)^{{n \choose r} - e(H)} p^{-t(H)} e^{ - \Lambda(n,r) + O(\omega \mu_{r+1} + n^{-\gamma})}$$
for every $H \in \cG$. 
\end{theorem}

Roughly speaking, Theorem~\ref{thm:Kr:intro} says that the copies of $K_r$ in $G(n,p)$ arise essentially independently, except for a bias towards edge-intersecting pairs of copies, with the weighting corresponding to how much more likely such pairs are in $G(n,p)$ than in $H_r(n,\pi)$. Indeed, note that the random hypergraph $H_r(n,\pi)$ is equal to $H$ with probability exactly $\pi^{e(H)} (1-\pi)^{{n \choose r} - e(H)}$, and the term $p^{-t(H)}$ compensates for the difference between $\binom{r}{2} e(H)$ and the number of edges of $G(H)$. The term $e^{-\Lambda(n,r)}$ is perhaps more surprising; to explain it, let $Z$ denote the number of pairs of edges of $H_r(n,\pi)$ that intersect in exactly two vertices, and note that the distribution of $Z$ is roughly Poisson with mean $\lambda=\E[Z]$. Moreover, if $Z$ were exactly Poisson, then a standard calculation would give 
$$\E\big[ p^{-Z} \big] = \exp\big( \E[Z/p] - \E[Z] \big) = e^{\lambda' -\lambda},$$
where $\lambda'$ is the expected number of pairs of copies of $K_r$ in $G(n,p)$ that overlap in a single edge. A similar argument applies for other intersection sizes, so the formula in the theorem corresponds (roughly) to the distribution of $H_r(n,\pi)$ reweighted by $p^{-t(H)}$, with $e^{-\Lambda(n,r)}$ as normalizing factor. Of course this is not anything like a proof of the result, just an argument that it is plausible. The content of the theorem is that one can, in fact, use this relatively simple modification of the binomial random hypergraph $H_r(n,\pi)$ as a substitute for the (much more complicated) distribution of $H_r(G(n,p))$, even when this distribution is not close to that of $H_r(n,\pi)$. In particular, note that if $p = o( n^{-2/r} )$
then $\mu_{r+1} = o(1)$, and therefore Theorem~\ref{thm:Kr:intro} gives an asymptotic formula for the probability of the event $\{ H_r(G(n,p)) = H \}$, and thus an essentially complete description of the distribution of $H_r(G(n,p))$, up to total variation distance $o(1)$.

Theorem~\ref{thm:Kr:intro} will be finally proved in Section~\ref{secKr}, as a consequence of two more technical results, Theorem~\ref{thmain}, which will give the upper bound, and Theorem~\ref{th2way}, which extends this to a two-sided bound. We remark that we will also prove a similar result when $r \in \{3,4\}$, but in these cases the error terms are somewhat more complicated, so we defer the statement to Theorem~\ref{thm:Kr}. 

\subsection{The number of $K_r$-factors in $G(n,p)$}\label{sec:intro:factors}

While describing the distribution of the copies of $K_r$ in $G(n,p)$ seems to us to be a natural aim in its own right, one might wonder whether Theorem~\ref{thm:Kr:intro} has any concrete applications. We will show that the answer is yes, considering the \emph{number} of $K_r$-factors in $G(n,p)$ above the threshold for such factors to appear. This threshold was determined up to a constant factor in a famous paper of Johansson, Kahn and Vu~\cite{JKV}, and it was recently shown that moreover 
\begin{equation}\label{p0def}
 q_r(n) = \big( (r-1)!\log n \big)^{1/\binom{r}{2}} n^{-2/r}
\end{equation}
is the sharp threshold for $G(n,p)$ to contain a $K_r$-factor; more precisely, this follows from the work of Kahn~\cite{Kahn-asymp} on perfect matchings in random hypergraphs, together with the couplings of the second author~\cite{R-copies} and Heckel~\cite{Heckel-copies}. 

There has been significant interest in recent years in the problem of bounding the number of $K_r$-factors in dense graphs, and related spanning structures such as perfect matchings and Hamilton cycles in dense hypergraphs, see for example~\cite{ABCDJMRS,FLM,JLS,KKKOP, KMP,KSW,MP,PSSS}. Perhaps surprisingly, this question turns out to be closely related to finding $K_r$-factors in $G(n,p)$, and several of these recent results have been deduced from generalisations of the Johansson--Kahn--Vu Theorem. In particular, Allen, Böttcher, Corsten, Davies, Jenssen, P.~Morris, Roberts and Skokan~\cite{ABCDJMRS} determined the threshold for the existence of a triangle-factor in a random subgraph of a graph $G$ with minimum degree at least $2n/3$, and used this result to deduce a bound on the number of triangle-factors in $G$ that is sharp up to a factor of roughly $(\log n)^{n/3}$. These results were then strengthened and generalised in~\cite{KMP,PSSS} using the recent breakthroughs on the Kahn--Kalai conjecture by Frankston, Kahn, Narayanan and Park~\cite{FKNP} and Park and Pham~\cite{PP}, together with the couplings from~\cite{Heckel-copies,R-copies}. For further background on clique factors in $G(n,p)$, see Section~\ref{secappl}.
 
Given $r \ge 3$ and a graph $G$, let us write $F_r(G)$ for the number of $K_r$-factors in $G$, let $H_r(n,m)$ denote the uniformly random $r$-uniform hypergraph with $n$ vertices and $m$ edges, and observe that if $n$ is a multiple of $r$, then the expected number of perfect matchings in $H_r(n,m)$ is exactly
\[
 \Sigma(n,m) := \frac{n!}{r!^{n/r}(n/r)!} \frac{(m)_{n/r}}{(N)_{n/r}},
\]
where $N=\binom{n}{r}$ and $(x)_k$ denotes the falling factorial $x(x-1)\cdots (x-k+1)$. We will prove the following theorem.

\begin{theorem}\label{thm:factorpweak}
Let $r \ge 3$ and $\eps > 0$ be constants, and suppose that $p = p(n)$ satisfies $p \ge (1+\eps) q_r(n)$ and $p=n^{-2/r+o(1)}$. If\/ $n \in \N$ and\/ $r \,|\, n$, then 
\begin{equation}\label{Fr1}
F_r(G(n,p)) = e^{o(n/\log n)} \, \Sigma(n,m)
\end{equation}
with high probability, where $m = \big\lfloor \binom{n}{r}p^{\binom{r}{2}} \big\rfloor$ is the expected number of copies of $K_r$ in $G(n,p)$ rounded to an integer.
\end{theorem}

The lower bound in~\eqref{Fr1} follows from the method of~\cite{Kahn-asymp} via the couplings from~\cite{Heckel-copies,R-copies} (see Corollaries~\ref{cor:chm} and~\ref{cor:factorl}); we will use Theorem~\ref{thm:Kr:intro} to prove (a stronger form of) the upper bound (see Theorem~\ref{thm:factorp}). It is perhaps helpful to note that the expectation of $F_r(G(n,p))$ is larger than $\Sigma(n,m)$ by a factor of $e^{\Theta(n/\log n)}$ when $p = \Theta(q_r(n))$ (see~\eqref{ratio} and~\eqref{EFr}). In particular, this means that, by Theorem~\ref{thm:factorpweak}, the number of $K_r$-factors in $G(n,p)$ is typically much smaller than its expectation when $p$ is close to the threshold. 


To put the error term in \eqref{Fr1} into context, we make the following conjecture. 

\begin{conjecture}\label{conjFr}
Let $r \ge 3$ be constant and let $p \le n^{-2/r + o(1)}$. Then
\begin{equation}\label{cFr}
 F_r(G(n,p)) = e^{-\Theta(n^3/m^2)} \, \Sigma(n,m) \cdot \1\big[ F_r(G(n,p)) > 0 \big]
\end{equation}
with high probability, where $m = \big\lfloor \binom{n}{r}p^{\binom{r}{2}} \big\rfloor$.
\end{conjecture}

Note that when $p = \Theta(q_r(n))$, the correction term in the conjecture is of the form $\exp\bb{-\Theta\big( n / (\log  n)^2 \big)}$; this term comes from the variation in the number of intersections between copies of $K_r$ in $G(n,p)$. We expect the conjecture to hold up to roughly $m = n^{4/3}$, at which point the variation in the number of copies of $K_r$ has a larger effect. 

We also expect a corresponding hitting time statement to hold. From the results of Kahn~\cite{Kahn-hitting} and Heckel, Kaufmann, M\"uller and Pasch~\cite{HKMP-hitting}, we know that with high probability the random graph process (adding edges randomly one by one) contains a $K_r$-factor as soon as every vertex is in a copy of $K_r$, and we conjecture that at this point the number of such $K_r$-factors is given by \eqref{cFr}. We also make the corresponding conjectures for Shamir's problem (perfect matchings in a random hypergraph); in that setting, a lower bound with a (larger) $\exp(o(n))$ error term is given in~\cite[Theorem 1.4]{Kahn-hitting}.

\subsection{Almost maximal cliques for constant $p$}\label{sscl2}


Our main theorems can also be applied to much larger cliques in denser random graphs; in fact, our original motivation for the results proved in this paper was to study the chromatic number of $G(n,1/2)$, where one needs to control the distribution of the independent sets of size close to $\alpha(G)$ (see, e.g.,~\cite{Heckel,HP,HR}). With that setting in mind, let us assume in this subsection that $p = p(n)$ is bounded away from zero and one, and that $p$ and $r = r(n) \sim 2\log_{1/p}(n)$ are chosen so that
\begin{equation}\label{mubeta}
 \mu_r = \Theta(n^{1+\theta}),
\end{equation}
where $\theta$ is a constant\footnote{More precisely, we will only consider those $n \in \N$ for which~\eqref{mubeta} holds, a set that we assume to be infinite. Let us also note that we don't actually need $\theta$ to be constant, but can allow it to vary in a compact interval.} and $\mu_r$ (as before) denotes the expected number of copies of $K_r$ in $G(n,p)$. For technical reasons, we will need to assume that $-1 < \theta < 1/2$; we are most interested in the case $\theta > 0$. Under this assumption $H_r(G(n,p))$ is unlikely to contain any pairs of copies of $K_r$ intersecting in $s$ vertices for any $3 \le s \le r - 2$, and will typically contain $\Theta^*(n^\theta)$ pairs that intersect in exactly $r - 1$ vertices.\footnote{Here, and throughout the paper, we write $f = O^*(g)$ to mean that there is a constant $C$ such that $f(n) = O\big( g(n) (\log n)^C \big)$, and $f = \Theta^*(g)$ if $f = O^*(g)$ and $g = O^*(f)$.} As a consequence, in this context we expect to be able to replace the correction term $\Lambda(n,r)$ in Theorem~\ref{thm:Kr:intro} by the simpler expression 
\begin{equation}\label{def:Lambda:prime}
\Lambda'(n,r) = \frac{1}{2} \binom{n}{r}\binom{r}{2}\binom{n-r}{r-2} p^{2\binom{r}{2}} \Big( p^{-1} - 1 \Big),
\end{equation}
which is just the term corresponding to $s = 2$ in~\eqref{def:Delta}, and indeed we obtain the following variant of Theorem~\ref{thm:Kr:intro} for constant $p$.

\begin{theorem}\label{thgen:intro}
Let $\eps > 0$ be constant, let $-1 + \eps \le \theta = \theta(n) \le 1/2 - \eps$, and suppose that $p = p(n) \in (\eps,1-\eps)$ and $r = r(n) \sim 2\log_{1/p}(n)$ are chosen so that $\mu_r = \Theta(n^{1+\theta})$. Then there exists a set $\cG$ of $r$-uniform hypergraphs with $\Pr\big( H_r(G(n,p)) \in \cG \big) \to 1$ as $n \to \infty$, such that  
$$\Pr\big( H_r(G(n,p)) = H \big) = \pi^{e(H)} (1-\pi)^{{n \choose r} - e(H)} p^{-t(H)} e^{-\Lambda'(n,r) + O^*(n^\theta) + n^{-\Omega(1)}}$$
for every $H \in \cG$. 
\end{theorem}

We will prove the upper bound in Section~\ref{ssGnp} for a simple and natural explicit family of graphs $\cG$ (see Definition~\ref{plausdef} and Theorem~\ref{thgen}), and the lower bound (which seems less important for applications) in Section~\ref{sslbg} for a different (and somewhat less natural) family $\cG$ (see Definition~\ref{defHgood} and Theorem~\ref{th2way2}).

\subsection{Relationship to previous work}

In~\cite{R-copies}, the second author studied, as here, the hypergraph $H_r(G(n,p))$ formed by the copies of $K_r$ in $G(n,p)$, for $r \ge 4$ and $p \le n^{-2/r+\gamma}$ for some $\gamma = \gamma(r) > 0$. The main result there (extended to $r=3$ by Annika Heckel in~\cite{Heckel-copies}) is that for some $\pi'\sim \pi=p^{\binom{r}{2}}$, there is a coupling of $H_r(G(n,p))$ and $H_r(n,\pi')$ such that 
$$H_r(n,\pi') \subset H_r(G(n,p))$$
with high probability; as noted in~\cite{R-copies}, there is no similar coupling in the other direction, precisely due to fact that $H_r(G(n,p))$ typically contains many more pairs of intersecting hyperedges than $H_r(n,\pi)$ does. 

To compare the two results: the coupling result of~\cite{R-copies} is useful for showing that certain structures exist in $G(n,p)$ (see, e.g.,~\cite{AKR,HKMP-hitting,JLS,KMP,PSSS}), and it is also relatively easy to apply, since it gives a direct comparison to a simple distribution. However, it is only `one-way', and loses something in that $\pi'$ is not equal to $\pi$. In contrast, Theorem~\ref{thm:Kr:intro} is two-way, and in a sense sharper, giving fairly precise bounds on the actual probability of all outcomes within some `typical' set. On the other hand, while in some sense it answers the distribution question for $H_r(G(n,p))$ almost completely, it is not at first clear how to apply it; we will illustrate one method in Section~\ref{secappl} (see in particular Lemma~\ref{lmoment}), which we expect to have further applications. The earlier coupling result definitely does not imply Theorem~\ref{thm:Kr:intro}; the reverse implication could perhaps hold, but seems hard to prove -- it is often difficult to go from probabilities of individual outcomes to coupling results. In summary, we believe that the two results are complementary.

Mousset, Noever, Panagiotou and Samotij~\cite{MNPS-non} proved results that are loosely related to those given here. They consider, for example, the probability that $G(n,p)$ contains no copies of some given graph $F$ (their actual context is much more general, as ours will be). For $F=K_r$, this is asking for the probability that $H_r(G(n,p))$ has no edges. They manage to give a very precise answer to this question in many cases. Here, in contrast, we consider not one possible outcome, but the entire set of `typical' outcomes. This complicates the situation considerably, and correspondingly our results are less precise.
We will briefly discuss the connection between the proof techniques used here and in~\cite{MNPS-non} once we have outlined our method (see Remark~\ref{rem_mnps}).

The rest of this paper is organized as follows: in Section~\ref{ssu} we describe the abstract set-up for our general results, and the basic method of proof. In Section~\ref{sec_wu}, we prove the simple `warm-up' result alluded to earlier. Then in Section~\ref{secmain} we prove the main part of our main result, Theorem~\ref{thmain}, giving an upper bound on the probability of suitable outcomes. In Section~\ref{seclower} we prove a corresponding lower bound with additional conditions. In Section~\ref{secKr} we consider the specific case of copies of $K_r$, in particular deducing Theorems~\ref{thm:Kr:intro} and~\ref{thgen:intro}. Finally, in Section~\ref{secappl} we turn to applications, proving Theorem~\ref{thm:factorpweak} via a general lemma that may be useful for applying Theorem~\ref{thm:Kr:intro} in other contexts.

\section{Set-up and basic approach}\label{ssu}

Almost all of our arguments can be phrased in a rather abstract setting, which we will do. For simplicity, as in the previous section, the reader may well wish to consider the case of copies of a given graph $H=H_n$ in $G(n,p)$, and more specifically the case $H=K_r$, where (for asymptotics) $r$ and $p$ are functions of $n$; we call this the `standard setting', and will often illustrate definitions by saying what they correspond to in this case.

\medskip\noindent
{\bf Set-up:} Let $\cH$ be an $s$-uniform hypergraph\footnote{In our standard setting $X = E(K_n)$ and $s = e(H)$, and the edges of $\cH$ correspond to the edge-sets of the copies of $H$ in $K_n$.} with vertex set $X$ and edge set $E(\cH) = \{ E_1,\ldots,E_N \}$. Write $X_p$ for the random subset of $X$ obtained by selecting each element independently with probability $p$, let $A_i$ be the event that $E_i \subset X_p$, and set $\pi = p^s$ and $\mu = \pi N$. We will assume that $1 - \pi = \Omega(1)$, and moreover that $\cH$ is \emph{symmetric}, in the sense that for each $i$ and $j$ there is an automorphism of $\cH$ that maps $E_i$ to $E_j$. 


Throughout we consider the random set 
$$I = \big\{ i \in [N] : A_i \text{ holds} \big\}.$$
Our aim is to compare the distribution of $I$ to that of the random subset $[N]_\pi$ of $[N]$. To do so, given a set $Y \subset [N]$ (which we think of as corresponding to `yes' outcomes), define 
$$R(Y) = \bigcup_{j \in Y} E_j$$
and for each $j \notin Y$, let $A_j' = A_j'(Y)$ denote the event that 
$$E_j \setminus R(Y) \subset X_p\setminus R(Y).$$
(In~\cite{R-copies}, the set $R(Y)$ was thought of as the set of  `revealed' elements.) 
Note that $I = Y$ if and only if $R(Y) \subset X_p$ and $A_j'(Y)^\cc$ holds for every $j \notin Y$, that is,
\[
 \big\{ I = Y \big\} = \big\{ R(Y) \subset X_p \big\} \cap \bigcap_{j \notin Y} A_j'(Y)^\cc.
\]
We may of course think of the $A_j'$ as events in the probability space corresponding to choosing a random subset of $X \setminus R(Y)$. Then
\[
 \Pr(I = Y) = p^{|R(Y)|} \cdot \Pr\Bb{ \bigcap_{j \notin Y} A_j'(Y)^\cc }.
\]
We bound the probability of the intersection by choosing an order $\prec$ on the index set $Y^\cc = [N] \setminus Y$, and setting
\[
 \pi_j = \Pr\Bb{ A_j'(Y) \Bigmid\; \bigcap_{i\prec j} A_i'(Y)^\cc },
\]
so that
\begin{equation}\label{PrE}
 \Pr(I=Y) = p^{|R(Y)|} \prod_{j \in Y^\cc} (1-\pi_j).
\end{equation}
Note that our order $\prec$ can (and will) depend on $Y$.

When evaluating conditional probabilities such as $\pi_j$, we use a fairly standard approach, similar to the Warnke version~\cite{RW-Janson} of the proof of the Janson inequalities~\cite{Janson}.
Let $U$ and $\{ U_i : i \in \cS \}$ be up-sets. Partition $\cS$ into $\cS_0 \cup \cS_1$ so that for all $i \in \cS_0$ the events $U$ and $U_i$ are independent.\footnote{Of course we can just take $\cS_0$ to be precisely these $i$, but it may be convenient to shrink $\cS_0$ and enlarge $\cS_1$. The argument only requires independence of certain events, never dependence.}
For $t \in \{0,1\}$, let 
$$D_t = \bigcap_{i \in \cS_t} U_i^\cc,$$ 
so $D_0$ and $D_1$ are down-sets, and $U$ is independent of $D_0$. The key formula is
\begin{equation}\label{mainbd0}
 \Pr\bigg( U \Bigmid \bigcap_{i \in \cS} U_i^\cc \bigg) = \Pr\big( U \mid D_0 \cap D_1 \big) =
 \frac{\Pr(U \cap D_1 \mid D_0)}{\Pr(D_1\mid D_0)} = \frac{\Pr(U)-\eps}{1-\delta},
\end{equation}
where 
\[
 \eps = \Pr\Bb{\bigcup_{i \in \cS_1} U \cap U_i \Bigmid D_0} \qquad \text{and} \qquad
 \delta = \Pr\Bb{\bigcup_{i \in \cS_1} U_i \Bigmid D_0}.
\]
Sometimes we use the simpler bound
\begin{equation}\label{simple0}
 \Pr(U) \ge \, \Pr\big( U \mid D_0 \cap D_1) \,\ge\, \Pr(U) - \eps \,\ge\, \Pr(U) - \sum_{i \in \cS_1} \Pr\big( U \cap U_i \big),
\end{equation}
where the last inequality holds by Harris' Lemma and the union bound. We mostly (but not always) apply these bounds directly to the events $A_i'(Y)$.  That is, when bounding $\pi_j$ for a fixed $j \in [N] \setminus Y$, we let $U =  A_j'(Y)$ and $\{ U_i : i \in \cS \} = \{ A_i'(Y) : i \prec j \}$. In this setting, we will find it convenient to define $i \sim j$ to mean that $i\ne j$ and $E_i \cap E_j \ne\emptyset$,\footnote{We could define $i \sim j$ if $E_i \setminus R(Y)$ and $E_j \setminus R(Y)$ intersect, but this doesn't seem to change things much, and we would then need separate notation elsewhere in the proof when considering whether $E_i$ and $E_j$ intersect for $j \in Y$.} 
and to write 
$$i \to j \quad \text{if} \quad i \prec j \,\text{ and }\, i \sim j \qquad \text{and} \qquad  i \not\to j \quad \text{if} \quad i \prec j \,\text{ and }\, i \not\sim j.\footnote{Here we are abusing notation slightly, since $\not\to$ is not the negation of $\to$.}$$
In particular, with this notation the quantities $\eps$ and $\delta$ above become 
$$\eps_j = \Pr\bigg( \bigcup_{i\to j} A_i'(Y) \cap A_j'(Y) \Bigmid D_0(j) \bigg) \quad \text{and} \quad \delta_j = \Pr\bigg( \bigcup_{i\to j} A_i'(Y) \Bigmid D_0(j) \bigg),$$
where
\begin{equation}\label{def:Delta0j}
D_0(j) = \bigcap_{i \not\to j} A_i'(Y)^\cc,
\end{equation}
and it follows from~\eqref{mainbd0} that
\begin{equation}\label{mainbd}
 \pi_j = \frac{\Pr(A_j'(Y)) - \eps_j}{1 - \delta_j}.
\end{equation}
Moreover, if we define
\begin{equation}\label{def:epshat:deltahat}
\heps_j = \sum_{i \to j}\, \Pr\big( A_i'(Y) \cap A_j'(Y) \big) \qquad \text{and} \qquad  \hdelta_j = \sum_{i \to j}\, \Pr\big( A_i'(Y) \big),
\end{equation}
then, by the union bound and Harris' Lemma, we have $\eps_j \le \heps_j$ and $\delta_j \le \hdelta_j$. In this setting,~\eqref{simple0} gives
\begin{equation}\label{simple}
 \Pr\big( A_j'(Y) \big) \ge \pi_j \ge \Pr\big( A_j'(Y) \big) - \eps_j \ge \Pr\big( A_j'(Y) \big) - \heps_j.
\end{equation}
Note also that 
\begin{equation}\label{eq:hateps:pihatdelta}
\heps_j \ge \pi\hdelta_j,
\end{equation}
by Harris's Lemma and since $\Pr\big( A_j'(Y) \big) \ge \pi$.  

\begin{remark}\label{rem_mnps}
Using~\eqref{mainbd}, bounds on $\eps_j$ and $\delta_j$ translate to bounds on $\pi_j$. In turn, we can improve on the simple upper bounds $\eps_j \le \heps_j$ and $\delta_j \le \hdelta_j$ by applying inclusion--exclusion to deal with the union, and~\eqref{mainbd0} to deal with conditioning. This process can be repeated to higher and higher order, eventually producing bounds where the error terms involve `clusters' of events $A_i'(Y)$ of arbitrarily large size. (Here a cluster is a set of events inducing a connected subgraph of the dependency digraph.)
Remarkably, Mousset, Noever, Panagiotou and Samotij~\cite{MNPS-non} managed to systematize something like this, giving an expansion for the probability that no member of a (suitable) family of events holds involving `cumulants' up to any given order, and bounds on the error. 

In principle we could try to apply their result to our collection of events $A_j'(Y)$.
However this does not seem feasible: the collection is rather hard to handle, with its properties depending on the set $Y$. Moreover, in our context, any event $A_j'(Y)$ that is affected by some $A_i$ with $i \in Y$ is already effectively a cluster of size $2$. Our plan here is essentially to expand the probability that none of the $A_j'(Y)$ holds to a formula including contributions from clusters of the original events $A_i$ of size 2, with error terms involving clusters of size 3. This involves treating different $A_j'(Y)$ differently. It is this grouping and the calculation that constitute the work here; the way we apply \eqref{mainbd} is relatively simple, only going one step further using \eqref{simple0} to bound $\eps_j$, and only in some cases. Our method and results therefore seem almost orthogonal to those of~\cite{MNPS-non}.
\end{remark}

\section{The warm up}\label{sec_wu}

In this section we will warm ourselves up for the main theorems by proving a weaker result, which will allow us to present some the key ideas without too many distracting technical complications. To state this result, assume that we are in the setting described in the previous section, recall that $\mu = \pi N$ denotes the expected number of events $A_i$ that hold, and set
\begin{equation}\label{Delta2def}
 \Delta_2 = \frac{1}{2} \sum_{i \sim j} \Pr\big( A_i \cap A_j \big),
\end{equation}
where the sum runs over all ordered pairs of distinct indices in $[N]$ with $i \sim j$. Thus, in the standard setting where the sets $E_i$ correspond to copies of a graph $H$, $\Delta_2$ is exactly the expected number of pairs of (edge-)overlapping copies. Recalling that $\cH$ is $s$-uniform, for any $Y \subset [N]$, let
\begin{equation}\label{tYdef}
 t(Y) = s|Y| - |R(Y)|
\end{equation}
be the number of repeated elements (edges in the standard setting) in the union $\bigcup_{i \in Y} E_i$, with an element (edge) appearing in $k$ sets $E_i$ counted $k - 1$ times.

We will prove the following weaker version of our main theorem. 

\begin{theorem}\label{th_wu}
For every $Y\subset [N]$ with $|Y| = O(\mu)$, we have
\begin{equation}\label{wu}
 \Pr( I = Y ) \le \pi^{|Y|} (1-\pi)^{N-|Y|} p^{-t(Y)} e^{O(\Delta_2)}.
\end{equation}
Moreover, if $\Delta_2 \to 0$ then
there is a set $\cT\subset \cP([N])$ of `typical' outcomes such that $\Pr(I\in\cT)\to 1$, and 
\begin{equation}\label{wu:hoo}
\Pr(I = Y) \sim \pi^{|Y|} (1 - \pi)^{N - |Y|}
\end{equation}
for each $Y\in \cT$. Equivalently, the random sets $I$ and $[N]_{\pi}$ can be coupled to agree with probability $1-o(1)$.
\end{theorem}

In the case of copies of some graph $H$ in $G(n,p)$, if we do not expect to see any pairs of copies intersecting even in single vertices, then a result such as Theorem~\ref{th_wu} is essentially trivial: the number $X$ of copies has essentially a Poisson distribution with mean $\mu$, and given that $X = k$, all outcomes consisting of $k$ vertex-disjoint copies are equally likely by symmetry. At first it might appear that this argument extends easily to the case $\Delta_2=o(1)$, but this is not obvious: it is true that for any two sets $Y_1$ and $Y_2$ of $k$ edge-disjoint copies, the events $Y_1\subset I$ and $Y_2 \subset I$ have exactly the same probability $\pi^k$. But it does not follow that $I=Y_1$ and $I=Y_2$ have the same probability: not all sets of $k$ edge-disjoint copies are isomorphic, and in general the conditional probability that no other copies are present will be different for $Y_1$ and $Y_2$.

Before starting the proof of Theorem~\ref{th_wu}, let us explain one key idea.
Given $Y\subset [N]$, i.e., a potential outcome for the random variable $I$, set
\begin{equation}\label{XYdef}
 L_2(Y) = \sum_{j \in Y} \sum_{i \sim j} \Pr\big( A_i \mid A_j \big),
\end{equation}
which is related to the extent to which the events that actually hold push up the probability of other events. Recalling that $j \in I$ if and only if $A_j$ holds, observe that 
\[
 L_2(I) = \sum_{j \in [N]} \ind{A_j} \sum_{i \sim j} \Pr\big( A_i \mid A_j \big).
\]
Since $\cH$ is symmetric (by assumption), the second sum in both formulae takes the same value ($c$, say) for every $j \in [N]$. Hence $L_2(Y)=c|Y|$ is proportional to $|Y|$. However, we also have
\begin{equation}\label{EXY}
 \Ex\big[ L_2(I) \big] = \sum_{j \in [N]} \Pr(A_j) \sum_{i \sim j} \Pr\big( A_i \mid A_j \big)
 = \sum_{i \sim j} \Pr\big( A_i \cap A_j \big) = 2\Delta_2
\end{equation}
and $\Ex[|I|] = \mu$, and therefore
\begin{equation}\label{XY1}
 L_2(Y) = \frac{2|Y|}{\mu} \cdot \Delta_2.
\end{equation}
We will use variants of this argument several times, sometimes with extra inequalities, such as the inequality $\Pr(A_i)\Pr(A_j) \le \Pr(A_i \cap A_j)$. In general, we call a sum such as $L_2(Y)$ a \emph{conditioned cluster expectation}: note that it is not the conditional expectation of the number of `clusters' (defined later), but is a sum of terms each of which can be thought of as a certain conditional expectation of a contribution to the number of clusters. We call the function $L_2$ because (i)~the clusters that are relevant here have size two, i.e., involve two of the events $A_i$, and (ii)~the dependence on $Y$ via indicator functions is linear.

\begin{proof}[Proof of Theorem~\ref{th_wu}]
We start with the first statement,~\eqref{wu}. Throughout the proof we fix the `target' outcome $Y \subset [N]$ with $|Y|=O(\mu)$, and set $m = |Y|$. Observe first that, by the definition~\eqref{tYdef} of $t(Y)$, we have
\[
 p^{|R(Y)|} = p^{s|Y| - t(Y)} = \pi^m p^{-t(Y)}. 
\]
In order to prove~\eqref{wu}, it will therefore suffice to show (see~\eqref{PrE}) that   
\[
  \prod_{j \in Y^\cc} (1 - \pi_j) \le (1 - \pi)^{N-m} e^{O(\Delta_2)}.
\]

Let us say that $j \in Y^\cc$ is \emph{neutral} if $E_j$ is disjoint from the set $R(Y)$, i.e., if there is no $i \in Y$ with $i \sim j$. Note that $A_j'(Y) = A_j$ if and only if $j$ is neutral. A key part of our strategy is to choose the ordering $\prec$ on the indices $j\in Y^\cc$ so that neutral ones come first.

Recall from~\eqref{def:epshat:deltahat} the definition of $\heps_j$, and that $\eps_j \le \heps_j$, and observe that if $j$ is neutral and $i \to j$, then $i$ is also neutral, since $i \prec j$ and by our choice of ordering. It follows that $\heps_j = \sum_{i \to j} \Pr(A_i \cap A_j)$, and hence 
\begin{equation}\label{sjne}
 \sjn \eps_j \le \sjn \heps_j\le \Delta_2.
\end{equation}
Hence, writing $N_0$ for the number of neutral $j$, it follows from~\eqref{simple} that
\[
 \pjn (1-\pi_j) \le \pjn (1-\pi+\heps_j) \le (1-\pi)^{N_0} e^{O(\Delta_2)},
\]
since $\Pr(A_j'(Y)) = \Pr(A_j) = \pi$ if $j$ is neutral, and $1 - \pi$ is bounded away from zero, by assumption. For non-neutral terms we simply use $\pi_j \ge 0$, to obtain
\begin{equation}\label{pjn0}
 \prod_{j\in Y^\cc} (1-\pi_j) \le  \pjn (1-\pi_j) \le (1-\pi)^{N_0} e^{O(\Delta_2)}.
\end{equation}
Now, by~\eqref{XYdef} and Harris' lemma we have
\begin{equation}\label{XY}
 L_2(Y) = \sum_{j \in Y} \sum_{i \sim j} \Pr\big( A_i \mid A_j \big) \ge \sum_{j \in Y}\sum_{i \sim j} \Pr(A_i) \ge \pi (N - m - N_0),
\end{equation}
where in the last step we used that fact that if $i \in Y^\cc$ is not neutral, then there exists some $j \in Y$ with $i \sim j$, so these $N - m - N_0$ terms each contribute at least $\Pr(A_i) = \pi$. Finally, recall from~\eqref{XY1} that $L_2(Y) = 2\Delta_2 |Y| / \mu = O(\Delta_2)$, since $|Y| = O(\mu)$, and hence
\begin{equation}\label{eq:need:footnote}
 (1-\pi)^{N - m - N_0} = e^{O(\Delta_2)},
\end{equation}
where we again used our assumption that $1 - \pi$ is bounded away from zero. Combining this with~\eqref{pjn0}, it follows that
\[
  \prod_{j\in Y^\cc} (1 - \pi_j) \le (1 - \pi)^{N-m} e^{O(\Delta_2)},
\]
as required.

In order to deduce the second part of the theorem, we need to show that if $\Delta_2=o(1)$, then with high probability $|I| = O(\mu)$ and $t(I) = 0$. The first of these follows\footnote{More precisely, Markov tells us that $|I| = o(\mu / \Delta_2)$ with high probability, which suffices to prove~\eqref{eq:need:footnote}, and hence~\eqref{wu}, up to a factor of $1 + o(1)$.} from Markov's inequality, since $\Ex[|I|] = \mu$, and the second holds because
$$\Pr\big( t(I) > 0 \big) \le \sum_{i \sim j} \Pr\big( A_i \cap A_j \big) = 2\Delta_2.$$
Since
$$\Pr\big( I = Y \big) \le \big( 1 + o(1) \big) \Pr\big( [N]_\pi = Y \big)$$
for every such $Y \subset [N]$, by~\eqref{wu}, it follows that~\eqref{wu:hoo} holds for every $Y$ in some family $\cT$ such that $I \in \cT$ with high probability. This in turn implies that the total variation distance between the two distributions is  $o(1)$, as claimed.
\end{proof}

We remark that under mild additional assumptions we can give an explicit lower bound on the probabilities of suitable `good' outcomes, essentially matching the upper bound in \eqref{wu}. But, surprisingly, doing so turns out to be a little awkward; we will give such a bound in a more complicated context later.

\section{Including the effect of 2-clusters}\label{secmain}

In this section we will prove our main technical theorem, which we will use (in Sections~\ref{secKr} and~\ref{secappl}) to prove the upper bounds in Theorems~\ref{thm:Kr:intro},~\ref{thm:factorpweak} and~\ref{thgen:intro}. In Section~\ref{seclower}, we will adapt the method used in this section to prove a two-way bound in a slightly more restricted setting. Throughout the next two sections, we will assume that we are in the general setting introduced in Section~\ref{ssu}. 

In order to state the main theorem of this section, we will need a few more definitions. 
First, we say that a set $S\subset [N]$ of indices forms a \emph{cluster} if the relation $i\sim j$ induces a connected graph on $S$. The \emph{size} of the cluster is $|S|$, and a \emph{$k$-cluster} is a cluster of size $k$. The cluster $S$ is \emph{present} in $Y$ if $S \subset Y$. With this terminology, $\Delta_2$ is simply the expected number of $2$-clusters present in $I$. More generally we write $W_k(Y)$ for the number of $k$-clusters that are present in $Y$, and define
\begin{equation}\label{Dkdef}
 \Delta_k = \Ex\big[ W_k(I) \big].
\end{equation}
Theorem~\ref{th_wu} studied the distribution of edges in the hypergraph $\cH[X_p]$
(in the standard setting, the distribution of copies of some graph $H$ in $G(n,p)$) in the case in which we either do not expect $2$-clusters, or do, but are willing to tolerate an error in our probability of the form $e^{O(\Delta_2)}$. Our aim in this section is to go one step further, taking full account of the effect of $2$-clusters, and giving error terms that depend on the numbers of $3$-clusters and $4$-clusters in $Y$.

In order to state the most general version of the theorem we will prove, we will need to define some further types of conditioned cluster expectation. Each of them will have the key property that we can bound their expectation in terms of $\Delta_k$ for some $k \in \N$. First, define
\begin{equation}\label{Q2def}
 Q_2(Y) = \frac{1}{2} \sum_{i \sim j} \1\big[ i, j \in Y \big],
\end{equation}
which is just the number of intersecting pairs of edges of $\cH[X_p]$ when $I = Y$. Note that  $Q_2(I) = W_2(I)$, and hence $\Ex\big[ Q_2(I) \big] = \Delta_2$. Next, let
\begin{equation}\label{Q3def}
 Q_3(Y) = \sum_{\text{3-clusters } \{i,j,k\}} \1\big[ i,j \in Y \big] \cdot \Pr\big( A_k \mid  A_i \big),
\end{equation}
and 
\begin{equation}\label{Q4def}
 Q_4(Y) = \sum_{\substack{i \sim i' \sim j' \sim j \\ i' \not\sim j, \, i \not\sim j'}} \1\big[ i,j \in Y \big] \cdot \Pr\big( A_{i'} \cap A_{j'} \mid A_i \cap A_j \big),
\end{equation}
where the sum is over 4-tuples of \emph{distinct} indices with the given intersection structure. We will show below that $\Ex\big[ Q_3(I) \big] = O(\Delta_3)$ and $\Ex\big[ Q_4(I) \big] = O(\Delta_4)$.

We will also need two further simple definitions in order to state the theorem. First, we will need the following slight modification of $\Delta_2$:
\begin{equation}\label{def:Delta20}
 \Delta_2^0 \,=\, \frac{1}{2} \, \sum_{i \sim j} \, \Pr(A_i) \cdot \Pr(A_j).
\end{equation}
Note that this corresponds to the expected number of pairs of intersecting copies \emph{if there were no dependence between copies}. In other words, $\Delta_2^0$ is just the expected number of pairs with $i \sim j$ in the random set $[N]_{\pi}$. We will also need to make a (very weak) assumption on the quantity 
\begin{equation}\label{phidef}
\phi = \max_{i \ne j} \, \Pr\big( A_i \mid A_j \big).
\end{equation}
We are now ready to state our main upper bound on $\Pr( I = Y )$. 


\begin{theorem}\label{thmain}
Suppose that $\phi \Delta_2 = O(\Delta_3)$. Then, for every $Y\subset [N]$ with $|Y| = O(\mu)$, we have
$$\Pr( I = Y ) \le \pi^{|Y|} (1-\pi)^{N - |Y|} p^{-t(Y)} \exp\bigg( - \bigg( \frac{2|Y|}{\mu} - 1 \bigg) \big( \Delta_2 - \Delta_2^0 \big) + \eta(Y) \bigg),$$
where
\[ 
\eta(Y) = O\Big( \Delta_3 + \phi \cdot Q_2(Y) + Q_3(Y) + Q_4(Y) + \pi^2 N \Big). 
\]
\end{theorem}

Since the statement of Theorem~\ref{thmain} is rather complicated, a quick sanity check is perhaps called for:~note that taking $Y = \emptyset$, we obtain
$$\Pr( I = \emptyset ) \le (1-\pi)^{N} \exp\Big( \Delta_2 - \Delta_2^0 + O\big( \Delta_3 + \pi^2 N \big) \Big),$$
by~\eqref{tYdef},~\eqref{Q2def},~\eqref{Q3def} and~\eqref{Q4def}. Since the cumulants (defined in~\cite[Section~1.3]{MNPS-non}) satisfy $\kappa_1 = \mu$ and $\kappa_2 = \Delta_2 - \Delta_2^0$, this does indeed match (in this special case) the result of Mousset, Noever, Panagiotou and Samotij (see~\cite[Theorem~11]{MNPS-non}), although the error terms are different. 

\begin{remark}\label{rmk:phiD2D3}
It will be easy to show that the bound $\phi \Delta_2 = O(\Delta_3)$ in Theorem~\ref{thmain}  holds in all of our applications. Indeed, it suffices that for any $E_i$ and $E_j$ with a given intersection size, there is another $E_k$ intersecting $E_j$ in the same number of elements. Then for each term $\Pr(A_i \cap A_j)$ contributing to $\Delta_2$, by symmetry we find some $k \ne i$ with $\Pr(A_k \,|\, A_j) = \phi$, and obtain a contribution of $\Pr(A_i \cap A_j \cap A_k) \ge \Pr(A_i \cap A_j)\Pr(A_k \,|\, A_j)$ to the sum defining $\Delta_3$. 
\end{remark}


To make sense of the error term in Theorem~\ref{thmain}, let us first prove the bounds claimed above on the expectation of $Q_i(I)$. 

\begin{lemma}\label{lem:ExQi}
$\Ex\big[ Q_i(I) \big] = O(\Delta_i)$ for each $i \in \{2,3,4\}$. 
\end{lemma}

\begin{proof}
As noted above, it follows immediately from~\eqref{Q2def} that $Q_2(I) = W_2(I)$, so $\Ex\big[ Q_2(I) \big] = \Delta_2$. Next, observe that 
\begin{align}\label{eq:ExQ3}
\Ex\big[ Q_3(I) \big] & \, = \sum_{\text{3-clusters } \{i,j,k\}} \Pr\big( A_i \cap A_j \big) \cdot \Pr\big( A_k \mid A_i \big) \nonumber \\
& \, \le \sum_{\text{3-clusters } \{i,j,k\}}\Pr\big( A_i \cap A_j \big) \cdot \Pr\big( A_k \mid A_i \cap A_j \big) = O(\Delta_3),
\end{align}
by Harris' Lemma and the definition~\eqref{Dkdef} of $\Delta_3$. Finally, note that
\begin{equation}\label{eq:ExQ4}
\Ex\big[ Q_4(I) \big] = \sum_{\substack{i \sim i' \sim j' \sim j \\ i' \not\sim j, \, i \not\sim j'}} \Pr\big( A_i \cap A_j \cap A_{i'} \cap A_{j'} \big) = O(\Delta_4),
\end{equation}
as required, since each term in the sum corresponds to a 4-cluster.
\end{proof}


Using Lemma~\ref{lem:ExQi}, we obtain the following corollary of Theorem~\ref{thmain} with more usable bounds. Recall that $\omega = \omega(n) \to \infty$ as $n \to \infty$ arbitrarily slowly.

\begin{corollary}\label{cormain}
Suppose that\/ $\phi \Delta_2 = O(\Delta_3)$, and that also $\mu \ge \omega^2$, $\Delta_2 = O(\mu)$ and $\Delta_4 = O(\Delta_3)$. Then there exists a collection $\cG$ of `good' outcomes $Y \subset [N]$, with $\Pr\big( I \in \cG \big) \to 1$ as $n \to \infty$, such that 
\begin{equation}\label{prss}
\Pr\bb{ I = Y } \le \pi^{|Y|} (1-\pi)^{N-|Y|} p^{-t(Y)} e^{-\Delta_2 + \Delta_2^0 + O(\omega \Delta_2 \mu^{-1/2} + \omega \Delta_3 + \pi^2 N)}
\end{equation}
for every $Y \in \cG$. Moreover, we can choose $\cG$ to be the set of\/ $Y \subset [N]$ such that \begin{equation}\label{def:G:cormain}
\big| |Y| - \mu \big| \le \omega \sqrt{\mu} \qquad \text{and} \qquad Q_i(Y) \le \omega \cdot \Ex\big[ Q_i(I) \big]
\end{equation}
for each $i \in \{2,3,4\}$.
\end{corollary}

\begin{proof}
Note first that the random set $I$ satisfies the conditions~\eqref{def:G:cormain} with high probability by Chebyshev's inequality and Markov's inequality respectively, since $|I|$ has expectation $\mu$ and variance at most 
\begin{equation}\label{eq:easy:variance:I}
\mu + \sum_{i \sim j} \Pr\big( A_i \cap A_j \big) = \mu + 2\Delta_2 = O(\mu).
\end{equation}
Now, let $\cG$ be the set of $Y \subset [N]$ satisfying~\eqref{def:G:cormain}, and observe that $|Y|=O(\mu)$ for every $Y \in \cG$, since $\mu \ge \omega^2$. Observe also that 
$$\bigg( \frac{2|Y|}{\mu} - 1 \bigg) \big( \Delta_2 - \Delta_2^0 \big) = \Delta_2 - \Delta_2^0 + O\bigg( \frac{\omega \Delta_2}{\sqrt{\mu}} \bigg)$$
for every $Y \in \cG$, since $\Delta_2^0 \le \Delta_2$, by Harris' Lemma. To deduce~\eqref{prss} from Theorem~\ref{thmain}, observe that 
\begin{equation}\label{eq:Qs:bound}
\phi \cdot Q_2(Y) + Q_3(Y) + Q_4(Y) = O( \omega \Delta_3)
\end{equation}
for every $Y \in \cG$, by Lemma~\ref{lem:ExQi} and~\eqref{def:G:cormain}, and hence that
$$- \bigg( \frac{2|Y|}{\mu} - 1 \bigg) \big( \Delta_2 - \Delta_2^0 \big) + \eta(Y) = - \Delta_2 + \Delta_2^0 + O\bigg( \frac{\omega \Delta_2}{\sqrt{\mu}} + \omega \Delta_3 + \pi^2 N \bigg)$$
for every $Y \in \cG$. The bound~\eqref{prss} therefore follows from Theorem~\ref{thmain}.
\end{proof}

Before diving into the details of the proof of Theorem~\ref{thmain}, let us give a brief sketch of the argument. Roughly speaking, our aim is to separate the effect of $2$-clusters, which contribute to the main correction term, from that of larger clusters, which we will throw into the error term. To do so, it will be useful to partition the set $Y^\cc$ as follows: we say that an index $j \notin Y$ is \emph{neutral} if there is no $i \in Y$ with $i \sim j$, that $j$ is \emph{simple} if there is exactly one $i \in Y$ with $i \sim j$, and that $j$ is \emph{complex} otherwise. In evaluating the probability that none of the events $A_j'(Y)$ holds as a product of conditional probabilities as in~\eqref{PrE}, we define an order $\prec$ on the indices $j \in Y^\cc$ so that neutral ones come first, then simple ones, then complex ones. We will show that
\begin{equation}\label{eq:main:neutral:contribution}
\pjn (1-\pi_j) \le \exp\bigg( - \mu + \bigg( \frac{2|Y|}{\mu} - 1 \bigg) \Delta_2^0 + \Delta_2 + O\big( \Delta_3 + \pi^2 N \big) \bigg)
\end{equation}
and 
$$\pjs (1-\pi_j) = \exp\bigg( - \frac{2|Y|}{\mu} \cdot \Delta_2 + O\Big( \Delta_3 + \phi \cdot Q_2(Y) + Q_3(Y) + Q_4(Y) \Big) \bigg).$$
Together with the trivial bound $\pi_j \ge 0$ for every complex $j$, by~\eqref{PrE} this will turn out to be sufficient to prove the theorem. 

The following basically trivial lemma will be used in the proof below.

\begin{lemma}\label{sumdown}
Let $D$ be a down-set, and define $d(A)=\Pr(A)-\Pr(A\mid D)$ for any event $A$. Then, if\/ $U_1,\ldots,U_k$ are up-sets, we have
\[
 0 \le d(U_1\cup\cdots\cup U_k) \le d(U_1)+\cdots+d(U_k).
\]
\end{lemma}

\begin{proof}
The first inequality is just Harris' Lemma. For the second,
starting with the identity $\Pr(U_1\cup U_2) = \Pr(U_1) + \Pr(U_2) - \Pr(U_1\cap U_2)$ and subtracting the same identity for probabilities conditioned on $D$, we obtain
\[
 d(U_1\cup U_2) = d(U_1)+d(U_2) - d(U_1\cap U_2) \le d(U_1)+d(U_2),
\]
since $U_1\cap U_2$ is an up-set. This proves the case $k=2$. The general case now follows easily by induction on $k$.
\end{proof}

The rest of this section is devoted to the proof of Theorem~\ref{thmain}.

\begin{proof}[Proof of Theorem~\ref{thmain}]
Fix a possible outcome $Y \subset [N]$ with $|Y| = O(\mu)$, and partition $Y^\cc$ into neutral, simple and complex terms as defined above. Fix an ordering $\prec$ of the elements of $Y^\cc$ such that neutral terms come first, then simple ones, then complex ones. Since the set $Y$ will be fixed throughout the proof, we will write $A_j'$ for $A_j'(Y)$. Our first main aim is to prove~\eqref{eq:main:neutral:contribution}, which bounds the contribution of the neutral terms to~\eqref{PrE}. 

\medskip\noindent
{\bf Neutral terms:}
For neutral terms we use~\eqref{mainbd}, which implies that
\begin{equation}\label{pijlo}
 \pi_j \ge (\pi-\eps_j)(1+\delta_j) = \pi - \eps_j + \pi\delta_j - \eps_j\delta_j,
\end{equation}
since if $j$ neutral then $\Pr(A_j') = \Pr(A_j) = \pi$. We will consider the sum of each of these terms one by one.  

\begin{claim}\label{Abd}
$$\sjn \pi \ge \mu - \frac{2|Y|}{\mu} \cdot \Delta_2^0 + O(\pi^2N).$$
\end{claim}

\begin{clmproof}{Abd}
Recall that if $j\in [N]$ is not neutral, then either $j \in Y$ or there exists $i \in Y$ with $i \sim j$. Since $\cH$ is symmetric and $\Pr(A_j) = \pi$ for every $j \in [N]$, it follows that
\begin{equation}\label{eq:Abd:basic}
 \sum_{j\in [N] \text{ not neutral}}  \pi \, \le \, \pi |Y| +\, \sum_{i \in Y} \sum_{j \sim i} \, \Pr(A_j) = c |Y|
\end{equation}
for some $c \ge 0$. Since $\Ex[|I|] = \mu$ and 
$$\Ex\bigg[  \sum_{i \in I} \sum_{j \sim i} \Pr(A_j) \bigg] = \Ex\bigg[ \sum_{i \sim j} \1[A_i] \cdot \Pr(A_j) \bigg] = \sum_{i \sim j} \Pr(A_i) \cdot \Pr(A_j) = 2\Delta_2^0,$$
it follows (as in the proof of~\eqref{XY1}) that $c = \pi + 2\Delta_2^0 / \mu$, and hence
$$\sjn \pi \ge \pi N - \bigg( \pi + \frac{2\Delta_2^0}{\mu} \bigg) |Y| = \mu - \frac{2|Y|}{\mu} \cdot \Delta_2^0 + O(\pi^2N),$$
as claimed, since $|Y| = O(\mu)$ and $\mu = \pi N$.
\end{clmproof}

We now turn to the sum of the $\eps_j$. In the present argument we only need an upper bound, which is simply $\Delta_2$ as shown in \eqref{sjne}. However, we will need a lower bound later, and the same method applies (in a less intuitive way) to the sum over $\pi\delta_j$, for which we need a lower bound, so we give the details.

\begin{claim}\label{Bbd}
$$\sjn \eps_j = \Delta_2 + O(\Delta_3).$$
\end{claim}

\begin{clmproof}{Bbd}
Recall (see Section~\ref{ssu}) the definition of $\eps_j$, and that we write $i \to j$ if both $i \sim j$ and $i \prec j$. Noting that $\Pr(A_j') = \Pr(A_j)$ if $j \in Y^\cc$ is neutral, observe that
\begin{equation}\label{eq:Bbd:split}
\sjn \eps_j = B_0 - B_1 - B_2,
\end{equation}
where
$$B_0 = \sjn \, \sum_{i \to j} \, \Pr\big( A_i \cap A_j \big),$$
is the main term, 
$$B_1 = \sjn \Pr\bigg( \bigcup_{i \to j} \big( A_i \cap A_j \big) \bigg) - \sjn \Pr\bigg( \bigcup_{i \to j} \big( A_i \cap A_j \big) \Bigmid D_0(j) \bigg)$$
is the effect of conditioning on the down-sets $D_0(j)$ (see~\eqref{def:Delta0j}), and 
$$B_2 = \sjn \, \sum_{i \to j} \, \Pr\big( A_i \cap A_j \big) - \sjn \Pr\bigg( \bigcup_{i \to j} \big( A_i \cap A_j \big) \bigg)$$
is the slack in the union bound. We will show that
\begin{equation}\label{B0bd}
B_0 = \Delta_2 + O(\Delta_3),
\end{equation}
and that $B_1$ and $B_2$ are both $O(\Delta_3)$. 

To prove~\eqref{B0bd}, note first that $B_0$ is the sum of $\Pr(A_i \cap A_j)$ over unordered pairs $\{i,j\}$ such that $i\sim j$ and $i$ and $j$ are both neutral (by our choice of $\prec$), while $\Delta_2$ is the same sum without the restriction to neutral pairs. It follows that
$$\Delta_2 - \sum_{k \in Y} \sum_{i \sim j \sim k} \Pr\big( A_i \cap A_j \big) \le B_0 \le \Delta_2,$$
since if $i$ and $j$ are not both neutral, then either one of them (wlog $i$) is in $Y$ (note that we allow $k = i$ in the second sum above), or there exists $k \in Y$ with (wlog) $j \sim k$. By symmetry, the double sum is (as in the previous claim) equal to $c|Y|$ for some $c \ge 0$, and 
$$\Ex\bigg[ \sum_{k \in I} \sum_{i \sim j \sim k} \Pr\big( A_i \cap A_j \big) \bigg] \le \sum_{i \sim j \sim k} \Pr(A_k) \cdot \Pr\big( A_i \cap A_j \big) = O(\Delta_3),$$
since the terms with $k = i$ contribute at most $\pi \Delta_2 \le \phi \Delta_2 = O(\Delta_3)$, those with $k \ne i$ correspond to 3-clusters, and $\Pr(A_k) \cdot \Pr( A_i \cap A_j) \le \Pr( A_i \cap A_j \cap A_k)$, by Harris' Lemma. 
Since $\Ex[|I|] = \mu$, it follows that $c = O(\Delta_3 / \mu)$, and hence 
\begin{equation}\label{ijkY}
 \sum_{k \in Y} \sum_{i \sim j \sim k} \Pr\big( A_i \cap A_j \big) = O\bigg( \frac{\Delta_3 |Y|}{\mu} \bigg) =  O(\Delta_3),
\end{equation}
since $|Y| = O(\mu)$. We therefore have $B_0 = \Delta_2+O(\Delta_3)$, as claimed.

To bound $B_1$, recall that $D_0(j)$ is a down-set for each $j$, and therefore, by applying Lemma~\ref{sumdown} to the family of up-sets $\big\{ A_i \cap A_j : i \to j \big\}$ for each neutral $j$, we have
\begin{equation}\label{eq:sumdown:app}
 B_1 \le \sjn \,\sum_{i \to j}\, \Big( \Pr\big( A_i \cap A_j \big) - \Pr\big( A_i \cap A_j \mid D_0(j) \big) \Big).
\end{equation}
To bound the right-hand side, we now apply~\eqref{simple0} for each pair $i \to j$ such that $j$ is neutral, with $U = A_i \cap A_j$, $\cS = \big\{ \ell : \ell \not\to j \big\}$, $U_\ell = A_\ell$ for each $\ell \in \cS$, 
and the partition $\cS = \cS_0 \cup \cS_1$ given by setting
$$\cS_1 = \big\{ \ell \in \cS \,:\, \ell \sim i \big\}.$$ 
Note that this is a legal partition, since $\ell \in \cS_0$ implies $\ell \not\sim j$ and $\ell \not\sim i$, so the events $A_i \cap A_j$ and $A_\ell$ are independent, as required. Recalling that $D_0(j)$ is the intersection of the events $\big\{ A_\ell^\cc : \ell \not\to j \big\}$, it follows from~\eqref{simple0} and~\eqref{eq:sumdown:app} that\footnote{Indeed, by~\eqref{simple0} we have $\Pr(U) - \Pr\big( U \mid D_0 \cap D_1) \le \sum_{\ell \in \cS_1} \Pr\big( U \cap U_\ell \big)$, which is equivalent to the claimed inequality $\Pr( A_i \cap A_j ) - \Pr( A_i \cap A_j \mid D_0(j) ) \le \sum_{\ell \in \cS_1} \Pr(A_i \cap A_j \cap A_\ell )$.} 
\begin{equation}\label{eq:B1bd}
 0 \le B_1 \le \sjn \,\sum_{i \to j} \, \sum_{\substack{\ell \not\to j \\ \ell \sim i}} \, \Pr\big(A_i \cap A_j \cap A_\ell \big).
\end{equation}
Each triple $(i,j,\ell)$ with $i \to j$, $\ell \not\to j$ and $\ell \sim i$ corresponds to a 3-cluster, since $\ell \sim i \sim j$, and $\ell \not\to j$ implies that $j \ne \ell$, so it follows from~\eqref{eq:B1bd} that $B_1 = O(\Delta_3)$, as claimed. 

Finally, to bound $B_2$ (the loss in the union bound), simply observe that 
$$0 \le B_2 \le \sjn \, \sum_{i \to j} \, \sum_{\substack{\ell \to j\\ \ell \prec i}} \, \Pr\big(A_i \cap A_j \cap A_\ell) = O(\Delta_3),$$
by inclusion--exclusion, and since each triple $(i,j,\ell)$ in the sum corresponds to a 3-cluster, since $i \sim j \sim \ell$ and $\ell \ne i$. Combining this with~\eqref{eq:Bbd:split} and~\eqref{B0bd} and our bound on $B_1$, the claim follows.
\end{clmproof}

We can bound the sum of $\pi\delta_j$ over neutral $j$ in the same way, using the fact that $\pi \cdot \Pr(A_i) \le \Pr(A_i \cap A_j)$. Since the proof of the following claim is almost identical to that of Claim~\ref{Bbd}, we will be somewhat briefer with the details. 

\begin{claim}\label{Cbd}
$$\pi \sjn \delta_j = \Delta_2^0 + O(\Delta_3).$$
\end{claim}

\begin{clmproof}{Cbd}
Recalling the definition of $\delta_j$, and that $\Pr(A_j') = \Pr(A_j) = \pi$ if $j \in Y^\cc$ is neutral, observe that
\begin{equation}\label{eq:Cbd:split}
\pi \sjn \delta_j = C_0 - C_1 - C_2,
\end{equation}
where
$$C_0 = \sjn \, \sum_{i \to j} \, \pi^2,$$
is the main term, 
$$C_1 = \pi \sjn \Pr\bigg( \bigcup_{i \to j} A_i \bigg) -\pi \sjn \Pr\bigg( \bigcup_{i \to j} A_i \Bigmid D_0(j) \bigg)$$
is the effect of conditioning on the down-sets $D_0(j)$, and 
$$C_2 = \pi \sjn \, \sum_{i \to j} \, \Pr(A_i) - \pi \sjn \Pr\bigg( \bigcup_{i \to j} A_i \bigg)$$
is the slack in the union bound. We will show that
\begin{equation}\label{C0bd}
C_0 = \Delta_2^0 + O(\Delta_3),
\end{equation}
and that $C_1$ and $C_2$ are both $O(\Delta_3)$. To prove~\eqref{C0bd}, note first that if we dropped the neutrality condition we would have exactly $\Delta_2^0$. To show that $\Delta_2^0 - C_0 = O(\Delta_3)$, we simply repeat the proof of~\eqref{B0bd}. Indeed, we have 
$$\Delta_2^0 - \sum_{k \in Y} \sum_{i \sim j \sim k} \pi^2 \le C_0 \le \Delta_2^0,$$
and since $\Pr(A_i\cap A_j)\ge \pi^2$, by Harris' Lemma, it follows from~\eqref{ijkY} that the double sum is $O(\Delta_3)$, giving $\Delta_2^0 - C_0 = O(\Delta_3)$, as claimed.

We can also bound $C_1$ and $C_2$ as we did for $B_1$ and $B_2$. Indeed, applying Lemma~\ref{sumdown} to the family of up-sets $\big\{ A_i : i \to j \big\}$ for each neutral $j$, we have
$$C_1 \le \pi \sjn \,\sum_{i \to j}\, \Big( \Pr(A_i) - \Pr\big( A_i \mid D_0(j) \big) \Big).$$
Moreover, applying~\eqref{simple0} for each pair $i \to j$ such that $j$ is neutral with $U = A_i$, and with the set $\cS$, the up-sets $U_\ell$, and the partition $\cS = \cS_0 \cup \cS_1$ as before, we obtain
$$0 \le C_1 \le \pi \sjn \,\sum_{i \to j} \, \sum_{\substack{\ell \not\to j \\ \ell \sim i}} \, \Pr\big( A_i \cap A_\ell \big).$$
Since $\Pr\big( A_j \mid A_i \cap A_\ell \big) \ge \pi$, we have $\pi \cdot \Pr\big( A_i \cap A_\ell \big) \le \Pr\big(A_i \cap A_j \cap A_\ell \big)$, and so the bound $C_1 = O(\Delta_3)$ follows as before. 
Finally, observe that
$$0 \le C_2 \le \pi \sjn \, \sum_{i \to j} \, \sum_{\substack{\ell \to j\\ \ell \prec i}} \, \Pr\big(A_i \cap A_\ell) = O(\Delta_3),$$
by inclusion--exclusion, and since $\pi \cdot \Pr\big( A_i \cap A_\ell \big) \le \Pr\big(A_i \cap A_j \cap A_\ell \big)$, by Harris' Lemma, and each triple $(i,j,\ell)$ in the sum corresponds to a 3-cluster. Combining this with~\eqref{eq:Cbd:split} and~\eqref{C0bd} and our bound on $C_1$, the claim follows.
\end{clmproof}

Finally, we need to bound the sum of $\eps_j \delta_j$. 

\begin{claim}\label{Dbd}
$$\sjn \eps_j \delta_j \le \sjn\heps_j\hdelta_j = O(\Delta_3).$$
\end{claim}

\begin{clmproof}{Dbd}
Recall from~\eqref{def:epshat:deltahat} the definitions of $\heps_j$ and $\hdelta_j$, and that $\eps_j \le \heps_j$ and $\delta_j \le \hdelta_j$. To bound the second sum, recall that if $j$ is neutral and $i \to j$ then $\Pr(A_i') = \Pr(A_i)$ and $\Pr(A_i' \cap A_j') = \Pr(A_i \cap A_j)$, and therefore 
$$\heps_j\hdelta_j = \sum_{i \to j} \sum_{k \to j} \Pr\big( A_i \cap A_j \big) \cdot \Pr(A_k) \le \sum_{i \sim j} \sum_{k \sim j} \Pr\big( A_i \cap A_j \big) \cdot \Pr(A_k)$$
 for each neutral $j$. Summing over neutral $j$ now gives 
 $$\sjn \heps_j\hdelta_j \le \sum_j \sum_{i \sim j} \sum_{k \sim j} \Pr\big( A_i \cap A_j \big) \cdot \Pr(A_k) = O\big( \pi\Delta_2 + \Delta_3 \big) = O(\Delta_3),$$
since the terms with $k = i$ contribute $O(\pi\Delta_2) = O(\Delta_3)$, and the terms with $k \ne i$ contribute at most $O(\Delta_3)$, since $\Pr\big( A_i \cap A_j \big) \cdot \Pr(A_k) \le \Pr\big( A_i \cap A_j \cap A_k \big)$ and each triple corresponds to a 3-cluster. 
\end{clmproof}

Putting the pieces together, from~\eqref{pijlo} and Claims~\ref{Abd}--\ref{Dbd}, we have
$$\sjn \pi_j \ge \mu - \bigg( \frac{2|Y|}{\mu} - 1 \bigg) \Delta_2^0 - \Delta_2 + O(\Delta_3) +O(\pi^2 N),$$
and therefore 
\begin{equation}\label{prodneutral}
 \pjn (1-\pi_j) \le \exp\bigg( - \mu + \bigg( \frac{2|Y|}{\mu} - 1 \bigg) \Delta_2^0 + \Delta_2 + O\big( \Delta_3 + \pi^2 N \big) \bigg),
\end{equation}
as claimed in~\eqref{eq:main:neutral:contribution}. It remains to deal with the simple and complex $j$. 

\medskip\noindent
{\bf Simple terms:}
For simple $j$ we will be able to use~\eqref{simple} rather than \eqref{mainbd}, since a simple $j$ already corresponds (in some sense) to a cluster of size $2$, so we only need to evaluate $\sjs \pi_j$ to `first order'. Our main task will be to prove the following claim.

\begin{claim}\label{Sclm}
$$\sjs \pi_j = \frac{2|Y|}{\mu} \cdot \Delta_2 +O\Big( \Delta_3 + \phi \cdot Q_2(Y) + Q_3(Y) + Q_4(Y) \Big).$$
\end{claim}

\begin{clmproof}{Sclm}
Observe first that, by~\eqref{def:epshat:deltahat} and~\eqref{simple}, we have
\begin{equation}\label{eq:Sclm:firststep}
 \bigg| \sjs \pi_j - \sjs \Pr(A_j') \bigg| \le \sjs \heps_j = \sjs \, \sum_{i \to j} \, \Pr\big( A_i' \cap A_j' \big).
\end{equation}
Observe that if $j$ is simple then $\Pr(A_j') = \Pr\big( A_j \mid A_i \big)$ for the unique $i \in Y$ such that $i \sim j$. Thus, by~\eqref{XYdef} and~\eqref{XY1}, we have
\[
\sjs \Pr(A_j') \le \sum_{i \in Y} \sum_{j \sim i} \Pr\big( A_j \mid A_i \big) = L_2(Y) = \frac{2|Y|}{\mu} \cdot \Delta_2.
\]
To prove a lower bound, note that a pair $(i,j)$ with $i \in Y$ and $j \sim i$ contributes to $L_2(Y) - \sjs \Pr(A_j')$ if and only if at least one of the following happens: $j \in Y$, or there is some $k \in Y \setminus \{i,j\}$ such that $k \sim j$. Hence
\begin{equation}
0 \le L_2(Y) - \sjs \Pr(A_j') \le \sum_{\substack{i,j \in Y\\ i \sim j}} \Pr\big( A_j \mid A_i \big) + \sum_{\substack{i,k \in Y\\ i \ne k}} \sum_{i \sim j \sim k} \Pr\big( A_j \mid A_i \big).
\end{equation}
Now, observe that 
$$\sum_{\substack{i,j \in Y\\ i \sim j}} \Pr\big( A_j \mid A_i \big) = O\big( \phi \cdot Q_2(Y) \big),$$
by~\eqref{Q2def} and~\eqref{phidef}, and that 
$$\sum_{\substack{i,k \in Y\\ i \ne k}} \sum_{i \sim j \sim k} \Pr\big( A_j \mid A_i \big) \le \sum_{\text{3-clusters } \{i,j,k\}} \1\big[ i,k \in Y \big] \cdot \Pr\big( A_j \mid  A_i \big) = Q_3(Y),$$
by~\eqref{Q3def}. Recalling~\eqref{XY1}, it follows that
\begin{equation}\label{eq:Sbd}
\sjs \Pr(A_j') = \frac{2|Y|}{\mu} \cdot \Delta_2 + O\big( \phi \cdot Q_2(Y) + Q_3(Y) \big).
\end{equation}

To bound the right-hand side of~\eqref{eq:Sclm:firststep}, recall that if $j$ is simple and $i \to j$, then $i$ is either neutral or simple, by our choice of the ordering $\prec$. We first consider terms with $i$ neutral. In such terms $A_i' = A_i$ and $\Pr\big( A_i' \cap A_j' \big) = \Pr\big( A_i \cap A_j \mid A_k \big)$, where $k$ is the unique element of $Y$ with $j \sim k$. The total contribution of such terms is at most
$$\sum_{k \in Y} \sum_{\substack{i \sim j \sim k\\ i \ne k}} \Pr\big( A_i \cap A_j \mid A_k \big) = c |Y|$$
for some $c \ge 0$, by symmetry. Since $\Ex[|I|] = \mu$ and
$$\Ex\bigg[ \sum_{k \in I} \sum_{\substack{i \sim j \sim k\\ i \ne k}} \Pr\big( A_i \cap A_j \mid A_k \big) \bigg] = \sum_{\substack{i \sim j \sim k\\ i \ne k}} \Pr\big( A_i \cap A_j \cap A_k \big) = O(\Delta_3),$$
it follows that $c = O(\Delta_3/\mu)$, and hence 
\begin{equation}\label{eq:Sclm:usual:method}
\sum_{k \in Y} \sum_{\substack{i \sim j \sim k\\ i \ne k}} \Pr\big( A_i \cap A_j \mid A_k \big) = O\bigg( \frac{\Delta_3 |Y|}{\mu} \bigg) = O(\Delta_3),
\end{equation}
since $|Y| = O(\mu)$. 
On the other hand, if $i$ and $j$ are both simple, then there is a unique $\ell \in Y$ such that $i \sim \ell$. Here there are two cases. If $k = \ell$ then $\Pr\big( A_i' \cap A_j' \big) = \Pr\big( A_i \cap A_j \mid A_k \big)$, and since $i \sim j \sim k \sim i$ and $k \in Y$, the contribution is at most
$$\sum_{k \in Y} \sum_{\substack{i \sim j \sim k\\ i \ne k}} \Pr\big( A_i \cap A_j \mid A_k \big) = O(\Delta_3),$$
by~\eqref{eq:Sclm:usual:method}, while if $k \ne \ell$, then it is at most
\[
 \sum_{\substack{k,\ell \in Y \\ k \ne \ell}} \sum_{\substack{i,j \in Y^\cc \\ \ell \sim i \sim j \sim k\\ i \not\sim k, \, j \not\sim \ell}}  \Pr\big( A_i \cap A_j \mid A_k \cap A_\ell \big) \le Q_4(Y),
\]
by~\eqref{Q4def}. Hence
\[
\sjs \, \sum_{i \to j} \, \Pr\big( A_i' \cap A_j' \big) = O\big( \Delta_3 + Q_4(Y) \big).
\]
Combining this with~\eqref{eq:Sclm:firststep} and~\eqref{eq:Sbd}, we obtain
\begin{equation}\label{sjs}
 \sjs \pi_j = \frac{2|Y|}{\mu} \cdot \Delta_2 + O\Big( \Delta_3 + \phi \cdot Q_2(Y) + Q_3(Y)+ Q_4(Y) \Big),
\end{equation}
as claimed.
\end{clmproof}

It follows from Claim~\ref{Sclm} that
$$\pjs (1 - \pi_j) \le \exp\bigg( - \frac{2|Y|}{\mu} \cdot \Delta_2 + O\Big( \Delta_3 + \phi \cdot Q_2(Y) + Q_3(Y) + Q_4(Y) \Big) \bigg),$$
and combining this with~\eqref{prodneutral}, we obtain
$$\prod_{j\text{ neutral or simple}} (1-\pi_j) \le \exp\bigg( - \mu - \bigg( \frac{2|Y|}{\mu} - 1 \bigg) \big( \Delta_2 - \Delta_2^0 \big) + \eta(Y) \bigg),$$
where
\[
\eta(Y) = O\Big( \Delta_3 + \phi \cdot Q_2(Y) + Q_3(Y) + Q_4(Y) + \pi^2 N \Big). 
\]
Finally, using the trivial bound $\pjc (1-\pi_j) \le 1$, and noting that 
\[
 (1-\pi)^{N - |Y|} = \exp\big( - \pi N + \pi |Y| + O(\pi^2N) \big) = \exp\big( - \mu + O(\pi^2N) \big),
\]
and that $p^{|R(Y)|} = \pi^{|Y|} p^{-t(Y)}$, the bound in Theorem~\ref{thmain} follows from~\eqref{PrE}, completing the proof of the theorem. 
\end{proof}

\section{The lower bound}\label{seclower}

In this section we prove a lower bound on the probability $\Pr(I=Y)$ that a certain subset of our events $E_i$ hold and the rest do not (corresponding, in the standard setting, to a specific set of copies of $H$ being the ones that are present in $G(n,p)$). This bound will roughly match the upper bound in Theorem~\ref{thmain}, up to the error term. Here there is a surprising annoyance, caused by the `complex' terms $j$. Recall that $j\in Y^\cc$ is \emph{complex} if there are two or more elements $i \in Y$ such that $i \sim j$. Fixing $Y \subset [N]$, and defining $A_j' = A_j'(Y)$ as before, let
\begin{equation}\label{Cdef}
 C(Y) = \sjc \Pr(A_j') = \sjc p^{|E_j\setminus R(Y)|}.
\end{equation}
Then $C(I)$ is a random variable, and its expectation counts certain clusters. Unfortunately, we do not in general have good control on their size. Roughly speaking, we expect the main contribution to come from cases where there are exactly two elements $i \in Y$ with $i \sim j$; note that the expectation of the sum of such terms is $O(\Delta_3)$. However, we can't simply ignore additional intersections between $E_j$ and sets $E_i$ with $i \in Y$, since they may increase the conditional probability of $A_j'$.

We first state a result with an error term involving $C(Y)$. Then we shall discuss methods of bounding $C(Y)$. We need a further very mild assumption, namely that $Y$ is \emph{possible}, meaning that $\Pr(I = Y) > 0$. Note that this is equivalent to there not existing any $j \in Y^\cc$ with $E_j \subset R(Y)$.

\begin{theorem}\label{th2way}
Suppose that $1 - p = \Omega(1)$, $\Delta_2^0 \le \mu/4$, and $\phi \Delta_2 = O(\Delta_3)$. Then, for every possible $Y\subset [N]$ with $|Y| = O(\mu)$, we have
$$\Pr( I = Y ) = \pi^{|Y|} (1-\pi)^{N - |Y|} p^{-t(Y)} \exp\bigg( - \bigg( \frac{2|Y|}{\mu} - 1 \bigg) \big( \Delta_2 - \Delta_2^0 \big) + \eta'(Y) \bigg),$$
where
$$\eta'(Y) = O\Big( \Delta_3 + \phi \cdot Q_2(Y) + Q_3(Y) + Q_4(Y) + \pi^2 N + C(Y) \Big).$$
\end{theorem}

Note that the only differences from Theorem~\ref{thmain} are the additional (mild) assumptions that $Y$ is possible, $1 - p = \Omega(1)$ and $\Delta_2^0 \le \mu/4$, the two-way rather than one-way bound, and the appearance of $C(Y)$ in the error term $\eta'(Y)$. We remark that a corresponding two-way version of Corollary~\ref{cormain} also follows.

To prove Theorem~\ref{th2way}, we simply go through the proof of Theorem~\ref{thmain}, noting that almost all of the bounds proved there are two-way, and showing how to give lower bounds in the remaining places.

\begin{proof}[Proof of Theorem~\ref{th2way}] 
The first place in the proof of Theorem~\ref{thmain} where we gave a one-way bound was~\eqref{pijlo}. We claim that, using~\eqref{mainbd}, we can replace this by
\[
 \pi_j = \frac{\pi-\eps_j}{1-\delta_j} = (\pi - \eps_j)(1 + \delta_j) + O( \pi\delta_j^2 ).
\]
This will follow if we can show that 
$\delta_j \le 1/2$ for each neutral $j \in Y^\cc$. To see why this is true, note that if $j$ is neutral, then 
\[
\delta_j \le \hdelta_j = \sum_{i \to j} \, \Pr(A_i') = \sum_{i \to j} \pi \le \sum_{i \sim j} \pi = \frac{2\Delta_2^0}{\pi N} = \frac{2\Delta_2^0}{\mu} \le \frac{1}{2},
\]
where the first two steps follow from~\eqref{def:epshat:deltahat} and Harris' Lemma, the third holds because $i$ is neutral for every $i \prec j$, and hence $\Pr(A_i') = \Pr(A_i) = \pi$, the fifth holds by~\eqref{def:Delta20} and since $\cH$ is symmetric, and the last holds by assumption. 

It follows that the loss in~\eqref{pijlo} is $O(\pi\delta_j^2)$ for each neutral $j$, and by~\eqref{eq:hateps:pihatdelta} and Claim~\ref{Dbd}, the sum of these terms is at most
\[
 \sjn \pi\delta_j^2 \le \sjn \pi \hdelta_j^2 \le \sjn \heps_j\hdelta_j = O(\Delta_3).
\]
The only place in the proof of Claim~\ref{Abd} in which we gave a one-way bound was in~\eqref{eq:Abd:basic}; the loss here is due to terms $j$ that we have overcounted because either $j \in Y$ and there is some $i \in Y$ with $i \sim j$, or $j \in Y^\cc$ and there is more than one $i \in Y$ with $i \sim j$. The loss is therefore at most
\begin{equation}
\sum_{\substack{i,j \in Y \\ i \sim j}} \pi + \sum_{\substack{i,k \in Y \\ i \ne k}} \sum_{i \sim j \sim k} \pi = O\big( \phi \cdot Q_2(Y) + Q_3(Y) \big), 
\end{equation}
by~\eqref{Q2def} and~\eqref{Q3def}, and this is already included in the error term $\eta(Y)$. Since Claims~\ref{Bbd}--\ref{Sclm} all give two-way bounds, the only other loss for the neutral and simple terms is in the inequality $1 - \pi_j \le e^{-\pi_j}$, which we used twice, in~\eqref{prodneutral} and in the final calculation. In both cases, we can replace it by the two-way bound
\[
\log(1-\pi_j) = - \pi_j + O(\pi_j^2),
\]
which holds because $\pi_j = \pi$ for every neutral $j$, and $\pi_j \le \phi = \max_{i\ne j} \Pr\big( A_j \mid A_i \big)$ for every simple $j$, and $1 - \phi \ge 1 - p = \Omega(1)$, by assumption. The total loss in the inequality is therefore at most
$$O\big( \pi \mu + \phi \Delta_2 + \phi \eta(Y) \big) = O\big( \pi^2 N + \Delta_3 + \eta(Y) \big) = O(\eta(Y)),$$ 
since $|Y| = O(\mu)$ and $\phi \Delta_2 = O(\Delta_3)$, and also $\phi \le 1$.

What remains is to deal with complex $\pi_j$, where we previously just used the trivial inequality $ \pi_j \ge 0$. Here we simply note that from the first inequality in \eqref{simple} (which is simply Harris' Lemma), we have
\[
 \sjc \pi_j \le \sjc \Pr(A_j') = C(Y),
\]
by~\eqref{Cdef}. Now, since we only consider possible sets $Y$, we have $E_j \not\subset R(Y)$ for every $j \in Y^\cc$, and therefore  $\pi_j \le \Pr(A_j') = p^{|E_j\setminus R(Y)|} \le p$, which is bounded away from $1$, by assumption. It follows that $\log(1-\pi_j)=O(\pi_j)$,
so
\[
 \pjc (1-\pi_j) = e^{O(C(Y))},
\]
completing the proof.
\end{proof}

This result is of course rather a cop-out: we have expressed the final error term simply as what comes out of the proof. We now show that it is indeed useful, by giving two approaches to bounding $C(Y)$. The first is most useful for the case of subgraphs of constant size. Recall that $\cH$ is $s$-uniform.

\begin{lemma}\label{CYbd1}
Suppose that $s = O(1)$. Then
\[
 \Ex\big[ C(I) \big] = O\big( \Delta_3 + \cdots + \Delta_{s+1} \big).
\]
\end{lemma}

The key to the proof is the following definition.

\begin{definition}\label{defsc}
A \emph{star-cluster} $\cS$ is a pair $(j,T)$ where $j \in [N]$ and $T \subset [N]$ with $|T| \ge 2$ satisfy
$$E_j \not\subset R(T),$$
and if $|T| \ge 3$, then
$$E_i \cap E_j \not\subset R(T \setminus \{i\})$$
for every $i \in T$. 
\end{definition}

In words, the centre of the star $E_j$ is not contained in the union of the leaves $\{ E_i : i \in T \}$, and the set of leaves is minimal given $E_j \cap R(T)$, meaning that if $|T| > 2$, then each leaf intersects $E_j$ in some element that is not contained in any of the other leaves. Note that the definition implies that $|T| \le |E_j|$.

We will say that the star-cluster $\cS = (j,T)$ is \emph{pre-present} in $Y$ if $T \subset Y$, and \emph{present} in $Y$ if also $j \in Y$. Define
$$\pi_0(\cS) = p^{|R(T)|}, \qquad \pi_1(\cS) = p^{|E_j \cup R(T)|} \quad \text{and} \quad \pi_c(\cS) = p^{|E_j \setminus R(T)|},$$
so $\pi_0(\cS)$ is the probability that $\cS$ is pre-present in $I$, $\pi_1(\cS)$ is the probability that $\cS$ is present in $I$, and $\pi_c(\cS)$ is the probability that $\cS$ is present in $I$, conditioned on the event that $\cS$ is pre-present in $I$. Let us also define 
\begin{equation}\label{def:ChatY}
\hC(Y) = \sum_{\cS}  \1\big[ \cS\text{ is pre-present in } Y \big] \cdot \pi_c(\cS),
\end{equation}
and make the following simple but key observation. 

\begin{obs}\label{obs:CChat}
$C(Y) \le \hC(Y)$ for every possible\/ $Y \subset [N]$. 
\end{obs}

\begin{proof}
Fix $Y \subset [N]$, let $j$ be a complex term, and let $T$ be a minimal subset of $Y$ such that $E_j \cap R(Y) \subset R(T)$ and $|T| \ge 2$. Then $\cS = (j,T)$ is a star-cluster, since $E_j \not\subset R(Y)$ (because $Y$ is possible and $i \in Y^\cc$), and if $|T| \ge 3$, then $E_i \cap E_j \not\subset R(T \setminus \{i\})$ for every $i \in T$, by the minimality of $T$. Moreover, $\cS$ is pre-present in $Y$, since $T \subset Y$. Since $\Pr(A_j') = \pi_c(\cS)$, it follows that
$$C(Y) = \sum_{j\text{ complex}} \Pr(A_j') \le \sum_{\cS} \1\big[ \cS\text{ is pre-present in } Y \big] \cdot \pi_c(\cS) = \hC(Y),$$
as claimed.
\end{proof}

We can now easily deduce the claimed bound on $\Ex\big[ C(I) \big]$. 

\begin{proof}[Proof of Lemma~\ref{CYbd1}.]
Observe first that
\[
 \Ex\big[ \hC(I) \big] = \sum_{\cS} \pi_0(\cS)\pi_c(\cS) = \sum_{\cS} \pi_1(\cS),
\]
and that each $\cS$ corresponds to a $3$-cluster if $s = 2$, and a cluster of size between $3$ and $s$ if $s \ge 3$. By Observation~\ref{obs:CChat}, and since $s = O(1)$, it follows that
\[
\Ex\big[ C(I) \big] \le \Ex\big[ \hC(I) \big] = \sum_{\cS} \pi_1(\cS) = O\big( \Delta_3 + \cdots + \Delta_{s+1} \big),
\]
as claimed.
\end{proof}

At this point we only outline our second method for bounding $C(Y)$, since it is probably best illustrated by the concrete example in Section~\ref{ssGnp}. The idea is to consider a set $\cF$ of `forbidden configurations', which we will take to be certain types of cluster that are not likely to appear. Then we may define a `legal star-cluster' $(j,T)$ to be a star-cluster whose leaves do not form any configuration in $\cF$. In place of Observation~\ref{obs:CChat}, we obtain 
\[
  C(Y) \le \hC_{\cF}(Y) := \sum_{\cS\text{ legal}} \1\big[ \cS\text{ is pre-present in } Y \big] \cdot \pi_c(\cS)
\]
for every $\cF$-free possible set $Y \subset [N]$. 
If $I$ is with high probability $\cF$-free and $\Ex\big[ \hC_{\cF}(I) \big]$ is small, then we obtain good control on $C(I)$ with high probability.

\section{Cliques in $G(n,p)$}\label{secKr}

In this section we give two applications of our main results to cliques in $G(n,p)$, which in particular imply Theorems~\ref{thm:Kr:intro} and~\ref{thgen:intro}. Throughout this section, as in Subsections~\ref{sscl1} and~\ref{sscl2}, we consider copies of $K_r$ in $G(n,p)$, where 
$$3 \le r = r(n) \in \N \qquad \text{and} \qquad p = p(n) \in (0,1 - \eps)$$
for some constant $\eps > 0$. We will apply Theorems~\ref{thmain} and~\ref{th2way} to the ${r \choose 2}$-uniform hypergraph $\cH$ with vertex set $X = E(K_n)$ and $N = {n \choose r}$ edges $E_1,\ldots,E_N$ corresponding to the edge sets of copies of $K_r$. Observe that $X_p \sim E(G(n,p))$, $1 - \pi \ge 1 - p \ge \eps$, where $\pi = p^{r \choose 2}$, and that $\cH$ is symmetric (via permutations of the vertices of $K_n$), so $\cH$ satisfies the assumptions from the start of Section~\ref{ssu}.  

In addition to the ${r \choose 2}$-uniform hypergraph $\cH$, we will consider the (random) $r$-uniform hypergraph $H_r(G(n,p))$ that encodes the \emph{vertex sets} of the copies of $K_r$ in $G(n,p)$.\footnote{In the particular case of copies of $K_r$, it is more natural to record the set of vertices as a hyperedge, than the set of edges.} We will write $H$ for a possible outcome for $H_r(G(n,p))$ (so $H$ corresponds to the set $Y \subset [N]$), and emphasize that $H$ is now an $r$-uniform hypergraph with vertex set $V(K_n)$. For concreteness, let us write
\begin{equation}\label{def:YH}
Y(H) = \big\{ i \in [N] : V_i \in E(H) \big\},
\end{equation}
where $V_i$ is the vertex set of the copy of $K_r$ corresponding to $E_i \in E(\cH)$, and note that $Y(H)$ is possible if and only if $H = H_r(G)$ for some graph $G \subset K_n$. Let us also define 
\begin{equation}\label{def:CH:QiH}
C(H) = C(Y(H)) \qquad \text{and} \qquad Q_i(H) = Q_i(Y(H))
\end{equation}
for each $i \in \{2,3,4\}$, where $C(Y)$ and $Q_i(Y)$ were defined in~\eqref{Q2def},~\eqref{Q3def},~\eqref{Q4def} and~\eqref{Cdef}. Note that $t(H)$, defined in~\eqref{tdef}, is equal to $t(Y(H))$ (see~\eqref{tYdef}). Finally, recall that $\omega = \omega(n) \to \infty$ as $n \to \infty$ arbitrarily slowly.

\subsection{The case $r$ constant}\label{sec:constant:r:proof}

In this subsection we will prove Theorem~\ref{thm:Kr:intro}, and also (as promised in the introduction) a similar result when $r \in \{3,4\}$. Before stating it, let us first define the explicit family of `good' hypergraphs $\cG$ for which we will prove the claimed two-way bound. Recall that $\mu_r = {n \choose r} p^{r \choose 2}$.

\begin{definition}\label{defKrgood}
An $r$-uniform hypergraph $H$ with vertex set $V(K_n)$ is \emph{good} if $H = H_r(G)$ for some graph $G$, 
$$\big| e(H) - \mu_r \big| \le \omega \sqrt{\mu_r} \qquad \text{and} \qquad Q_i(H) \le \omega \cdot \Ex\big[ Q_i\big( H_r(G(n,p)) \big) \big]$$
for each $i \in \{2,3,4\}$, and moreover $C(H) \le \omega \cdot \Ex\big[ C\big( H_r(G(n,p)) \big) \big]$. 
\end{definition}

The following theorem immediately implies Theorem~\ref{thm:Kr:intro}. 

\begin{theorem}\label{thm:Kr}
For each $r \ge 3$, there exists $\gamma > 0$ such that if\/ $p \le n^{-2/r+\gamma}$ and $\mu_r \to \infty$, then 
\begin{equation}\label{eq:thm:Kr}
\Pr\big( H_r(G(n,p)) = H \big) = \pi^{e(H)} (1 - \pi)^{{n \choose r} - e(H)} p^{-t(H)} e^{- \Delta_2 +\Delta_2^0 + O(\omega \xi + n^{-\gamma})}
\end{equation}
for every good $H$, where
\begin{equation}\label{xidef}
\xi = \left\{ 
\begin{array}{crl}
n^5p^7 + \sqrt{n^5p^7} \quad & \text{if} \quad & r = 3 \\[+0.2ex]
n^3p^6 + n^8p^{16} \quad & \text{if} \quad & r = 4\\[+0.2ex]
\mu_{r+1} \quad & \text{if} \quad & r \ge 5.
\end{array} \right.
\end{equation}
Moreover, $H_r(G(n,p))$ is good with high probability.
\end{theorem}

To avoid repetition, let us fix $r \ge 3$ and $p = p(n) \le n^{-2/r+\gamma}$ throughout this subsection, where $\gamma > 0$ is sufficiently small, and let $\cG$ be the collection of good $r$-uniform hypergraphs with vertex set $V(K_n)$. Let us begin with the easiest task: showing that $H_r(G(n,p))$ is good with high probability. 

\begin{lemma}\label{lem:Hgood:whp}
$H_r(G(n,p))$ is good with high probability.
\end{lemma}

\begin{proof}
The bounds on $C(H)$ and $Q_i(H)$ hold with high probability by Markov's inequality, so it will suffice to prove the bound on $e(H)$. To do so, we will first show that $\Delta_2 = O(\mu_r)$. Indeed, we have
\begin{equation}\label{eq:Delta2:vs:mu}
 \frac{\Delta_2}{\mu_r} = \sum_{t = 2}^{r-1} O\big( n^{r-t} p^{\binom{r}{2}-\binom{t}{2}} \big) = O\big( n^{r-2} p^{\binom{r}{2} - 1} + np^{r-1} \big) = O\big( n^{2/r - 1 + O(\gamma)}\big),
\end{equation}
where the second step holds by convexity, and the third by our bound on $p$, which implies that $n^{r-2}p^{\binom{r}{2}-1} \le n^{2/r-1+O(\gamma)}$ and $np^{r-1} \le n^{2/r - 1 + O(\gamma)}$. Since $r \ge 3$, it follows from~\eqref{eq:Delta2:vs:mu} that if $\gamma$ is sufficiently small, then $\Delta_2 = O(\mu_r)$, as claimed. Now, recalling that $e(H_r(G(n,p)))$ has expectation $\mu_r$ and variance at most $\mu_r + 2\Delta_2$, by~\eqref{eq:easy:variance:I}, the lemma follows by Chebyshev's inequality.
\end{proof}

We will deduce Theorem~\ref{thm:Kr} from Theorem~\ref{th2way}; our only tasks will be to verify that the conditions of Theorem~\ref{th2way} are satisfied, and that the bound given by the theorem implies~\eqref{eq:thm:Kr} for all $H \in \cG$. To check that the conditions of Theorem~\ref{th2way} are satisfied, note first that
\begin{equation}\label{eq:th2way:conditions:check}
\phi \le p = o(1) \qquad \text{and} \qquad \Delta_2^0 = O\bigg( \frac{\mu_r^2}{n^2} \bigg) = o(\mu_r),
\end{equation}
since $\mu_r \le n^{r} p^{\binom{r}{2}} \le n^{1+O(\gamma)} = o(n^2)$ if $\gamma > 0$ is sufficiently small. Observe also that $\phi \Delta_2 = O(\Delta_3)$ holds by Remark~\ref{rmk:phiD2D3}. We may therefore apply Theorem~\ref{th2way} to obtain the following bound. 

\begin{lemma}\label{lem:bound1}
For every $H \in \cG$, we have
$$\Pr\big( H_r(G(n,p)) = H \big) = \pi^{e(H)} (1 - \pi)^{{n \choose r} - e(H)} p^{-t(H)} e^{- ( \frac{2e(H)}{\mu_r} - 1 ) ( \Delta_2 - \Delta_2^0 ) + \eta'(H)},$$
where
$$\eta'(H) = O\Big( \Delta_3 + \phi \cdot Q_2(H) + Q_3(H) + Q_4(H) + \pi^2 N + C(H) \Big).$$
\end{lemma}

\begin{proof}
We have already verified the general conditions of Theorem~\ref{th2way}. To check that the conclusion holds for the hypergraph $H \in \cG$, observe first that the set $Y(H)$ is possible, since $H = H_r(G)$ for some graph $G$. Moreover, if we choose $\omega = \omega(n)$ so that $\mu_r \ge \omega^2$, then $H \in \cG$ implies that $|Y(H)| = e(H) = O(\mu_r)$. The bound in Theorem~\ref{th2way} therefore holds for $Y(H)$, as claimed.
\end{proof}

Our next aim is to bound the error term $\eta'(H)$ for all $H \in \cG$. We start by comparing $\Delta_k$ for different $k$.

\begin{lemma}\label{DOD}
If\/ $2 \le k = O(1)$, then $\Delta_k = O(\Delta_{k-1})$, where $\Delta_1 = \mu_r$.
\end{lemma}

\begin{proof}
A $k$-cluster can be thought of (in at least one way) as arising from a $(k-1)$-cluster $\cC$ by adding a new copy $E_i$ of $K_r$ sharing some number $t \ge 2$ of vertices with $\cC$. Given $\cC$, there are at most $O(1)$ choices for the shared vertices, and at most $n^{r-t}$ choices for the new vertices in $\cC$. The conditional probability (given that $\cC$ is present) that $E_i$ is also present is then at most $p^{\binom{r}{2}-\binom{t}{2}}$, since $E_i$ can share at most $\binom{t}{2}$ edges with $\cC$. It follows that
\[
\frac{\Delta_k}{\Delta_{k-1}} = \sum_{t = 2}^r O\big( n^{r-t} p^{\binom{r}{2}-\binom{t}{2}} \big) = O\big( n^{r-2} p^{\binom{r}{2} - 1} + 1 \big) = O(1),
\]
as claimed, where the second step holds by convexity, and the final step holds because $n^{r-2}p^{\binom{r}{2}-1} \le n^{2/r-1+O(\gamma)} \le 1$ if $\gamma$ is sufficiently small.
\end{proof}

We can now bound most of the terms in $\eta'(H)$.

\begin{lemma}\label{lem:Qbounds:rconstant}
$$\phi \cdot Q_2(H) + Q_3(H) + Q_4(H) = O( \omega \Delta_3 )$$ 
for every $H \in \cG$.
\end{lemma}


\begin{proof}
By Lemma~\ref{lem:ExQi} and Definition~\ref{defKrgood}, we have 
$$Q_i(H) \le \omega \cdot \Ex\big[ Q_i\big( H_r(G(n,p)) \big) \big] = O( \omega \Delta_i)$$ 
for each $i \in \{2,3,4\}$. Since (as noted above) $\phi \Delta_2 = O(\Delta_3)$ and $\Delta_4 = O(\Delta_3)$, by Lemma~\ref{DOD}, the claimed bound follows.
\end{proof}

In order to obtain the required bound on $\eta'(H)$, it will now suffice to show that $\Delta_3 = O(\xi)$, where $\xi$ was defined in~\eqref{xidef}, since by Lemmas~\ref{CYbd1} and~\ref{DOD} this will imply the same bound for the expectation of $C(H_r(G(n,p)))$. 

In the case $r \ge 5$, we will actually obtain the following stronger and more general result, with essentially no extra effort. Recall that $r \in \N$ is fixed and $p \le n^{-2/r+\gamma}$ for some small $\gamma > 0$. 

\begin{lemma}\label{Krclust}
If $r \ge 5$ and $\cC$ is a $k$-cluster for some $3 \le k = O(1)$, then the expected number of copies of\/ $\cC$ present in $G(n,p)$ is $O(\mu_{r+1})$ if\/ $\cC$ has exactly $r+1$ vertices, and is $n^{-\Omega(1)}$ otherwise.
\end{lemma}

\begin{proof}
We list the hyperedges/copies of $K_r$ in $\cC$ in an order $E_1,\ldots,E_k$ such that each after the first shares at least one edge of $K_n$ with the union of the previous ones. For each $i \in [k]$, let $\cC_i$ be the $i$-cluster formed by the edges $E_1,\ldots,E_i$, and let $N_i$ denote the expected number of labelled copies of $\cC_i$ present in $G(n,p)$. Thus $N_1 = \mu_r$, and 
$$N_2 = O\big( \mu_r \cdot n^{r - t} p^{\binom{r}{2}-\binom{t}{2}} \big),$$ 
where $t$ is the number of vertices of $K_n$ that $E_1$ and $E_2$ share. Similarly, if $E_i$ shares $2 \le t \le r - 1$ vertices of $K_n$ with $E_1 \cup \cdots \cup E_{i-1}$, then from \eqref{eq:Delta2:vs:mu} we have
$$\frac{N_i}{N_{i-1}} = O\big( n^{r-t} p^{\binom{r}{2}-\binom{t}{2}} \big) = O\big( n^{2/r - 1 + O(\gamma)} \big),$$
while if $E_i$ shares $r$ vertices with $E_1 \cup \cdots \cup E_{i-1}$ then $N_i = O(N_{i-1})$, and if moreover $E_i$ contains $a$ new edges, then $N_i = O( p^a N_{i-1})$.

Let us say that step~$i$ is \emph{ordinary} if $E_i$ contains a vertex not in $E_1 \cup \cdots \cup E_{i-1}$. Note that step 2 is always ordinary; if there are at least two ordinary steps, then
$$N_k = O\big( \mu_r \cdot n^{4/r - 2 + O(\gamma)} \big) = O\big( n^{-1/5 + O(\gamma)} \big) = n^{-\Omega(1)},$$
since $\mu_r \le n^r p^{r \choose 2} \le n^{1 + O(\gamma)}$ and recalling that $r \ge 5$ and $\gamma$ is sufficiently small.

This leaves the case where $i = 2$ is the only ordinary step. Let $|E_1 \cup E_2| = 2r - t$, and suppose first that $t = r - 1$. Then there are $r+1$ vertices in total and (since $k \ge 3$ and the $E_i$ are distinct) $\cC$ must contain all of the edges between these vertices, so in this case $N_k = O(\mu_{r+1})$. On the other hand, if $t \le r - 2$, then $E_1 \cup E_2$ does not contain a copy of  $K_r^-$, so $E_3$ must contain at least two edges not in $E_1 \cup E_2$, and hence 
$$N_k = O(N_3) = O( p^2N_2 ) = O\big( p^2 \cdot \mu_r \cdot n^{2/r - 1 + O(\gamma)} \big) = n^{-\Omega(1)},$$ 
since $p \le n^{-2/r+\gamma}$, $\mu_r \le n^{1 + O(\gamma)}$ and $\gamma$ is sufficiently small.
\end{proof}

Before continuing, let's quickly note the following immediate corollary of Lemma~\ref{Krclust}. 

\begin{corollary}\label{newsmooth}
If $r \ge 5$ and $p \le n^{-2/r+\gamma}$, then with high probability every $k$-cluster of copies of $K_r$ in $G(n,p)$ with $k \ge 3$ forms a copy of $K_{r+1}$.
\end{corollary}

\begin{proof}
By Lemma~\ref{Krclust} and Markov's inequality, with high probability $G(n,p)$ does not contain any $k$-cluster with more than $r + 1$ vertices. Since every $k$-cluster with $r + 1$ vertices forms a copy of $K_{r+1}$, the corollary follows.
\end{proof}

We can now prove the claimed bound on $\Delta_3$. 

\begin{lemma}\label{lcKrclust}
$\Delta_3 = O(\xi)$. 
\end{lemma}

\begin{proof}
If $r = 3$, then any $3$-cluster is either a copy of $K_4$ or has exactly $5$ vertices and $7$ edges, so in this case
$$\Delta_3 = O\big( n^4p^6 + n^5p^7 \big) = O(\xi),$$
as claimed, since $\mu_r \to \infty$ implies that $pn \to \infty$. If $r = 4$, then one can easily check that any $3$-cluster is either a copy of $K_5$, or has $6$ vertices and at least $12$ edges, or $7$ vertices and at least $14$ edges, or $8$ vertices and at least $16$ edges. It follows that
$$\Delta_3 = O\big( n^5p^{10}+n^6p^{12} + n^7p^{14} + n^8p^{16} \big) = O(\xi),$$
where the final step follows by considering the cases $p^2n \le 1$ and $p^2 n \ge 1$ separately. We may therefore assume that $r \ge 5$, in which case our task is to show that $\Delta_3 = O(\mu_{r+1})$. We may also assume that $p = n^{-2/r+\gamma}$, since every $3$-cluster has at least $\binom{r+1}{2}$ edges, so $\Delta_3/\mu_{r+1}$ is increasing in $p$. It follows that $\mu_{r+1} \ge 1$, and therefore, by Lemma~\ref{Krclust}, that $\Delta_3 = O( \mu_{r+1} )$, as required.
\end{proof}

We are now ready to deduce Theorem~\ref{thm:Kr} from Theorem~\ref{th2way}.

\begin{proof}[Proof of Theorem~\ref{thm:Kr}]
Let $\cG$ be the set of good $r$-uniform hypergraphs from Definition~\ref{defKrgood}, and observe that $H_r(G(n,p)) \in \cG$ with high probability, by Lemma~\ref{lem:Hgood:whp}. 
Next, by Lemma~\ref{lem:bound1}, for every $H\in \cG$ we have
$$\Pr\big( H_r(G(n,p)) = H \big) = \pi^{e(H)} (1 - \pi)^{{n \choose r} - e(H)} p^{-t(H)} e^{- ( \frac{2e(H)}{\mu_r} - 1 ) ( \Delta_2 - \Delta_2^0 ) + \eta'(H)}.$$
Moreover, we have 
\begin{equation}\label{Krf0}
\eta'(H) = O\Big( \omega \cdot \Delta_3 + \pi^2 N + \omega \cdot \Ex\big[ C(H_r(G(n,p))) \big] \Big)
\end{equation}
for every $H \in \cG$, by Definition~\ref{defKrgood} and Lemma~\ref{lem:Qbounds:rconstant}. Moreover, it follows from Lemmas~\ref{CYbd1},~\ref{DOD} and~\ref{lcKrclust} that
\begin{equation}\label{Krf1}
\Delta_3 + \Ex\big[ C(H_r(G(n,p))) \big] = O(\Delta_3) = O(\xi),
\end{equation}
and it is easy to check that
\begin{equation}\label{Krf2}
\pi^2 N \le n^r p^{2{r \choose 2}} < n^{r+1} p^{{r+1 \choose 2}} = \mu_{r+1} \le \xi
\end{equation}
for all $r \ge 3$, and hence $\eta'(H)=O(\xi)$.

Finally, observe that 
\begin{equation}\label{Krf3}
\bigg( \frac{2e(H)}{\mu_r} - 1 \bigg) \big( \Delta_2 - \Delta_2^0 \big) = \Delta_2 - \Delta_2^0 + O\bigg( \frac{\omega \Delta_2}{\sqrt{\mu_r}} \bigg)
\end{equation}
for every $H \in \cG$, since $e(H) \in \mu_r \pm \omega \sqrt{\mu_r}$, by Definition~\ref{defKrgood}, and $\Delta_2 \ge \Delta_2^0$. It therefore only remains to show that 
\begin{equation}\label{ratxi}
\frac{\Delta_2}{\sqrt{\mu_r}} = O(\xi) + n^{-\Omega(1)}.
\end{equation}
When $r = 3$ this holds since $\Delta_2 = O(n^4p^5)$, $\mu_3 = \Theta(n^3p^3)$ and $\xi \ge \sqrt{n^{5}p^{7}}$, and when $r = 4$ since $\Delta_2 = O(n^5p^9 + n^6p^{11})$, $\mu_4 = \Theta(n^4p^6)$ and $\xi \ge n^3p^6 + n^4p^8$. Finally, when $r \ge 5$, it holds because, by \eqref{eq:Delta2:vs:mu} and since $\mu_r \le n^{1 + O(\gamma)}$, we have
\[
 \frac{\Delta_2}{\sqrt{\mu_r}} = O\big(n^{2/r-1+O(\gamma)} \sqrt{\mu_r} \big) = O\big( n^{2/r - 1/2 + O(\gamma)} \big) = n^{-\Omega(1)}
\]
for $\gamma$ sufficiently small. Combining~\eqref{ratxi} with~\eqref{Krf0},~\eqref{Krf1},~\eqref{Krf2} and~\eqref{Krf3}, we obtain the claimed bound, completing the proof of Theorem~\ref{thm:Kr}.
\end{proof}

To finish the subsection, let us quickly observe that Theorem~\ref{thm:Kr} implies Theorem~\ref{thm:Kr:intro}.

\begin{proof}[Proof of Theorem~\ref{thm:Kr:intro}]
To deduce Theorem~\ref{thm:Kr:intro} from Theorem~\ref{thm:Kr}, simply note that 
$$\Lambda(n,r) = \Delta_2 - \Delta_2^0,$$
by~\eqref{def:Delta},~\eqref{Delta2def} and~\eqref{def:Delta20}, and recall that $\xi = \mu_{r+1}$ when $r \ge 5$. 
\end{proof}

\begin{remark}\label{remomeg}
The use of Markov's inequality to bound the random variables $Q_i\big( H_r(G(n,p)) \big)$ and $C\big( H_r(G(n,p)) \big)$ by $\omega$ times their expectations is rather inefficient; in many cases it should be possible to show that these quantities are concentrated around their means. 
\end{remark}

\subsection{The case $p$ constant}\label{ssGnp}

In this subsection we will prove the upper bound in Theorem~\ref{thgen:intro}. Recall from the statement of the theorem that $r = r(n) \sim 2\log_{1/p}(n)$ for some $p$ bounded away from $0$ and $1$, and $\mu_r = \Theta(n^{1+\theta})$ for some $-1 + \eps \le \theta = \theta(n) \le 1/2 - \eps$. 

Given a hypergraph $H$, let us write $t_k(H)$ for the number of (unordered) pairs of hyperedges of $H$ sharing exactly $k$ vertices, and define 
\begin{equation}\label{nusdef}
\nu_k = \Ex\pig[ t_k\big( H_r(G(n,p)) \big) \pig] = \frac{1}{2}\binom{n}{r}\binom{r}{k} \binom{n-r}{r-k} p^{2\binom{r}{2}-\binom{k}{2}}
\end{equation}
and
\begin{equation}\label{nus0def}
\nu_k^0 = \Ex\big[ t_k\big( H_r(n,\pi) \big) \big] = p^{\binom{k}{2}} \nu_k. 
\end{equation}
Observe that $\Lambda'(n,r) = \nu_2 - \nu_2^0$, where $\Lambda'(n,r)$ was defined in~\eqref{def:Lambda:prime}.


We will prove that Theorem~\ref{thgen:intro} holds with $\cG$ equal to the following family of $r$-uniform hypergraphs. The purpose of the constants $C$ and $\delta$ in the definition is to allow us to obtain the conclusion of the theorem for as large a family of hypergraphs as possible. We will show below (see Lemma~\ref{lem:plausible:whp}) that $H_r(G(n,p))$ has the required properties with high probability for any $C > 0$ and $\delta \in (0,1/4)$. 

\begin{definition}\label{plausdef}
Let $C > 0$ and $\delta \in (0,1/4)$ be constants. We say that an $r$-uniform hypergraph $H$ with $n$ vertices is \emph{plausible} if the following hold:
\begin{equation}\label{plausible:tkHsmall}
 t_k(H) \le (\log n)^C \nu_k
\end{equation}
for $k \in \{0,1,2,r-1\}$,
\begin{equation}\label{plausible:tkHzero}
t_k(H) = 0
\end{equation} 
for $3 \le k \le r-2$, and
\begin{equation}\label{plausible:edges}
 |e(H) - \mu_r| \le n^{1-\delta}.
\end{equation}
\end{definition}

The following theorem implies the upper bound in Theorem~\ref{thgen:intro}. The error terms in the theorem depend on the constants $C$ and~$\delta$. 

\begin{theorem}\label{thgen}
Let $\eps > 0$ be constant, let $-1 + \eps \le \theta = \theta(n) \le 1/2 - \eps$, and suppose that $p = p(n) \in (\eps,1-\eps)$ and $r = r(n) \sim 2\log_{1/p}(n)$ are chosen so that~$\mu_r = \Theta(n^{1+\theta})$. Then 
$$\Pr\big( H_r(G(n,p)) = H \big) \le \pi^{e(H)} (1-\pi)^{{n \choose r} - e(H)} p^{-t(H)} e^{-\nu_2+\nu_2^0 + O^*(n^\theta) + n^{-\Omega(1)}}$$
for every plausible $H$. Moreover, $H_r(G(n,p))$ is plausible with high probability.
\end{theorem}

We will deduce Theorem~\ref{thgen} from Theorem~\ref{thmain}. In order to do so, observe first that  the condition $\phi \Delta_2 = O(\Delta_3)$ in the theorem is satisfied, by Remark~\ref{rmk:phiD2D3}. We may therefore apply Theorem~\ref{thmain} to the ${r \choose 2}$-uniform hypergraph encoding copies of $K_r$ in $E(K_n)$. 

Let us assume, until the end of the proof of Theorem~\ref{thgen}, that $\eps > 0$ is fixed, and that $p \in (\eps,1-\eps)$, $r \in \N$ and $\theta \in \R$ satisfy the assumptions of the theorem. Recall from~\eqref{def:YH} and~\eqref{def:CH:QiH} the definitions of $Y(H)$ and $Q_i(H)$. In order to deduce the claimed bound, we need to bound the error term $\eta(Y(H))$ for each plausible hypergraph $H$. We will prove the following bounds. 

\begin{lemma}\label{lem:pconstant:error:bound}
If $H$ is plausible, then 
\begin{equation}\label{eq:QiH:bound:pconstant}
Q_i(H) = O^*(\Delta_i)
\end{equation}
for each $i \in \{2,3,4\}$. Moreover
\begin{equation}\label{eq:Delta34:bounds:pconstant}
\Delta_3 = O^*(n^\theta)+n^{-\Omega(1)} \qquad \text{and} \qquad \Delta_4 = O^*(n^\theta) + n^{-\Omega(1)}.
\end{equation}
\end{lemma}

Recall from Lemma~\ref{lem:ExQi} that $\Ex\big[ Q_i\big( H_r(G(n,p)) \big) \big] = O(\Delta_i)$ for each $i \in \{2,3,4\}$, and from Definition~\ref{plausdef} that $t_k(H) \le (\log n)^C \nu_k$ for all plausible $H$ and all $0 \le k \le r - 1$. To prove~\eqref{eq:QiH:bound:pconstant}, it will therefore suffice to prove the following lemma.  

\begin{lemma}\label{quadcor}
If\/ $t_k(H) \le \lambda \cdot \nu_k$ for all\/ $0 \le k \le r - 1$, then 
$$Q_i(H) \le \lambda \cdot \Ex\big[ Q_i\big( H_r(G(n,p)) \big) \big]$$
for each $i \in \{2,3,4\}$.
\end{lemma}

\begin{proof}
The key observation is that, for fixed $n$ and $r$, there exist non-negative constants $c_i(k)$ such that 
\[ 
 Q_i(H) = \sum_{k = 0}^{r-1} c_i(k) \cdot t_k(H)
\]
for every $r$-uniform hypergraph $H$ on $n$ vertices. To see this, observe that by symmetry, in the sums defining $Q_i(H)$ (see~\eqref{Q2def},~\eqref{Q3def} and~\eqref{Q4def}) all pairs of $r$-cliques with a given intersection size $s$ contribute the same when we look at the terms where the indicator functions correspond to these $r$-cliques, and moreover this contribution does not depend on $H$. Since $t_k(H) \le \lambda \cdot \nu_k$ for all $0 \le k \le r - 1$, the claimed bound follows from the definition~\eqref{nusdef} of $\nu_k$.
\end{proof}

\begin{remark}
Note that Lemma~\ref{quadcor} holds for any $n$, $r$ and $p$, and as a consequence it can also be used when verifying that a hypergraph is good, in the sense of Definition~\ref{defKrgood}.
\end{remark}

In order to prove~\eqref{eq:Delta34:bounds:pconstant}, we will consider separately those clusters that are contained in sets of size $r+1$, and those that are not. For each $k \ge 3$, let us say that a $k$-cluster is \emph{normal} if it has at least $r+2$ vertices in total.

\pagebreak

\begin{lemma}\label{G34normal}
If $k \in \{3,4\}$, then the expected number of normal $k$-clusters is 
$$O^*(n^{3\theta - 1}) + n^{-\Omega(1)}.$$
\end{lemma}

In the proof of Lemma~\ref{G34normal}, and also later in the section, it will be convenient to define
\[
 \hb = \max\{\theta,0\}.
\]
We will also need the following simple observation. 

\begin{obs}\label{obs:p:tothe:r}
If $\mu_r = n^{O(1)}$, then $p^r = \Theta^*(n^{-2})$. 
\end{obs}

\begin{proof}
Observe that 
$$\mu_r = {n \choose r} p^{{r \choose 2}} = n^{O(1)} \bigg( \frac{n}{r} \cdot p^{r/2} \bigg)^r,$$
and hence $p^{r/2} = \Theta(r/n) = \Theta^*(n^{-1})$, as claimed. 
\end{proof}

The proof of the lemma is somewhat similar to that of Lemma~\ref{Krclust}. 

\begin{proof}[Proof of Lemma~\ref{G34normal}]
Observe first that the number of isomorphism classes of $k$-clusters is $O^*(1)$, so it will suffice to prove the claimed bound for a fixed $k$-cluster $\cC$. Let $E_1,\ldots,E_k$ be the cliques in $\cC$, ordered so that each after the first shares an edge of $K_n$ with the union of the earlier ones. 
The complication in this setting is that, since $r$ is no longer constant, there are many ways to choose the subset of the vertices of the previous cliques to re-use in $E_i$. Indeed, a crude bound on $m_i$, the number of choices in step $i \le 4$, is $2^{3r} = n^{O(1)}$, where the implicit constant depends only on $\eps$, since $r \sim 2\log_{1/p}(n)$ and $p \in (\eps,1-\eps)$. 

To deal with this issue, let us say that two cliques have a \emph{strange} intersection if they share between $\lambda$ and $r - \lambda$ vertices, where $\lambda$ is a suitable (sufficiently large) constant. As before, let $N_i$ denote the expected number of labelled copies of $\cC_i$ (the $i$-cluster formed by the cliques $E_1,\ldots,E_i$) present in $G(n,p)$.
If two cliques in $\cC$ have a strange intersection of size $t$, then we may list them first, and deduce that
$$N_2 \le \mu_r \cdot {n \choose r - t} p^{\binom{r}{2} - \binom{t}{2}} \le n^{-\lambda/2},$$ 
since $p^r = \Theta^*(n^{-2})$, by Observation~\ref{obs:p:tothe:r}. Now, observe that $$N_i = O^*\big( m_i N_{i-1} \big),$$
as in the proof of Lemma~\ref{Krclust}, since if $E_i$ shares $2 \le t \le r - 1$ vertices of $K_n$ with $E_1 \cup \cdots \cup E_{i-1}$, then
\begin{equation}\label{eq:Ni:vs:miNi-1}
\frac{N_i}{m_i N_{i-1}} \le {n \choose r - t} p^{\binom{r}{2}-\binom{t}{2}} = O\big( \max\big\{ \mu_r r^2 / n^2, \, p^r n \big\} \big) = O^*(n^{\hb - 1}),
\end{equation}
by Observation~\ref{obs:p:tothe:r}, while if $E_i$ shares $r$ vertices with $E_1 \cup \cdots \cup E_{i-1}$ then $N_i \le m_i N_{i-1}$. Using our crude bound on $m_i$, it follows that for such clusters we have
$$N_k \le n^{-\lambda/2 + O(1)} \le n^{-\lambda/4}$$
if we choose $\lambda$ sufficiently large (depending on $\eps$). 

We may therefore assume that $\cC$ contains no strange intersections. But now when choosing vertices to re-use, from each earlier clique we must pick either at most $\lambda$ vertices, or all but at most $\lambda$ vertices. Thus $m_i = O((\log n)^{\lambda})$ for each $i$, and hence the total number of choices for the shared vertices is $O^*(1)$. In other words, we are now in a situation where we can ignore these choice factors.

We can now complete the proof by repeating the argument from Lemma~\ref{Krclust}. Indeed, step 2 is always ordinary (meaning that $E_i$ contains a vertex not in $E_1 \cup \cdots \cup E_{i-1}$), and if there are at least two ordinary steps, then
$$N_k \le O^*\big( \mu_r \cdot n^{2\hb - 2} \big) = O^*( n^{\theta + 2\hb - 1} ) \le O^*( n^{3\theta - 1} ) + n^{-\Omega(1)},$$
by~\eqref{eq:Ni:vs:miNi-1} and since $\mu_r = \Theta( n^{1 + \theta} )$. On the other hand, if $i = 2$ is the only ordinary step, then $N_k = O^*(N_2)$. If $E_1$ and $E_2$ share $3 \le t \le r - 2$ vertices, then
$$N_2 \le \mu_r \binom{r}{t} {n \choose r - t} p^{\binom{r}{2} - \binom{t}{2}} = O^*\big( \max\big\{ \mu_r^2 / n^3, \, \mu_r p^{2r} n^2 \big\} \big) \le n^{-\Omega(1)},$$ 
since $\mu_r = O(n^{3/2-\eps})$, and so we are done. If they share only $2$ vertices, then $E_3$ must have at least $r - 3$ edges outside $E_1 \cup E_2$, so 
$$N_3 = O^*\big( p^{r-3}N_2 \big) = O^*\big( n^{-2} \cdot \mu_r^2 / n^2 \big) = n^{-\Omega(1)},$$
by Observation~\ref{obs:p:tothe:r}, as required. Finally, if they share $r - 1$ vertices then, since no other steps are ordinary, our cluster is not normal. 
\end{proof}

We can now easily deduce the claimed bounds on $Q_i(H)$ and $\Delta_i$. 

\begin{proof}[Proof of Lemma~\ref{lem:pconstant:error:bound}]
Let $H$ be plausible, and recall from Definition~\ref{plausdef} that $t_k(H) = O^*(\nu_k)$ for all $0 \le k \le r - 1$. By Lemmas~\ref{lem:ExQi} and~\ref{quadcor}, it follows that 
$$Q_i(H) = O^*\big( \Ex\big[ Q_i\big( H_r(G(n,p)) \big) \big] \big) = O^*(\Delta_i)$$
for each $i \in \{2,3,4\}$, as claimed. 

To bound $\Delta_k$ for $k \in \{3,4\}$, note first that, by Lemma~\ref{G34normal} and since $\theta \le 1/2 - \eps$, the expected number of normal $k$-clusters is at most $O^*(n^\theta) + n^{-\Omega(1)}$. Moreover, every non-normal $k$-cluster forms a copy of $K_{r+1}$ in $G(n,p)$. Since the expected number of copies of $K_{r+1}$ in $G(n,p)$ is
\begin{equation}\label{eq:murplusone}
\mu_{r+1} = {n \choose r+1} p^{{r+1 \choose 2}} \le p^r n \cdot \mu_r = O^*( n^{\theta} ),
\end{equation}
by Observation~\ref{obs:p:tothe:r}, and each copy of $K_{r+1}$ contains $O^*(1)$ different $k$-clusters, it follows that 
$$\Delta_k \le O^*( n^{\theta} ) + n^{-\Omega(1)},$$
as required.
\end{proof}

Finally, we need to show that $H_r(G(n,p))$ is plausible with high probability. 

\pagebreak

\begin{lemma}\label{lem:plausible:whp}
$H_r(G(n,p))$ is plausible with high probability.
\end{lemma}

\begin{proof}
We start by bounding $\nu_k$, defined in~\eqref{nusdef}, first noting that
$$\nu_3 = O\big( \mu_r^2 \cdot (r^2 / n)^3 \big) = O^*(n^{2\theta-1}) \quad \text{and} \quad \nu_{r-2} = O\big( \mu_r \cdot r^2 n^2 p^{2r} \big) = O^*(n^{\theta-1}),$$
since $\mu_r = \Theta(n^{1+\theta})$ and $p^r = O^*(n^{-2})$, by Observation~\ref{obs:p:tothe:r}. By convexity, and since $\theta\le 1/2-\eps$, it follows that $\nu_k = n^{-\Omega(1)}$ for all $3 \le k \le r-2$. Since $\nu_k = \Ex\big[ t_k\big( H_r(G(n,p)) \big) \big]$,
the bounds~\eqref{plausible:tkHsmall} and~\eqref{plausible:tkHzero} thus both hold for $H_r(G(n,p))$ with high probability by Markov's inequality. To show that~\eqref{plausible:edges} also holds with high probability, it will suffice to show that $\Delta_2 = O(\mu_r)$, since $e(H)$ has expectation $\mu_r = O(n^{3/2})$ and variance at most $\mu_r + 2\Delta_2$, by~\eqref{eq:easy:variance:I}. To see this, note that 
$$
 \nu_2 = O\big( \mu_r^2 \cdot (r^2 / n)^2 \big) = O^*(n^{2\theta}) \quad \text{and} \quad \nu_{r-1} = O\big( \mu_r \cdot r n p^{r} \big) = O^*(n^{\theta})
$$
so
\begin{equation}\label{eq:Delta2:sumnus}
\Delta_2  = \sum_{k = 2}^{r-1} \nu_k = \nu_2 + O^*\big( n^\theta+n^{2\theta-1} \big) =  O^*\big( n^{2\theta} + n^{\theta} \big),
\end{equation}
and recall that $\mu_r = \Theta(n^{1+\theta})$ and $\theta \le 1/2$. By Chebyshev's inequality, it follows that~\eqref{plausible:edges} holds with high probability, as required.
\end{proof}

We are now ready to prove Theorem~\ref{thgen}.

\begin{proof}[Proof of Theorem~\ref{thgen}]
Let $H$ be a plausible $r$-uniform hypergraph, and note that $e(H) = O(\mu_r)$, by~\eqref{plausible:edges} and since $\mu_r = \Theta(n^{1+\theta})$ and $\theta \ge -1+\eps$. Since (as noted above) $\phi \Delta_2 = O(\Delta_3)$, we may apply Theorem~\ref{thmain}, and deduce that
$$\Pr\big( H_r(G(n,p)) = H \big) \le \pi^{e(H)} (1-\pi)^{N - e(H)} p^{-t(H)} e^{- ( \frac{2e(H)}{\mu_r} - 1 ) ( \Delta_2 - \Delta_2^0 ) + \eta(H)},$$
where $N = {n \choose r}$ and
\[ 
\eta(H) = O\Big( \Delta_3 + \phi \cdot Q_2(H) + Q_3(H) + Q_4(H) + \pi^2 N \Big). 
\]
Moreover, by Lemma~\ref{lem:pconstant:error:bound}, and since $\pi^2 N = n^{-\Omega(1)}$, we have
$$\eta(H) = O^*\big( \Delta_3 + \Delta_4 + n^{-\Omega(1)} \big) = O^*(n^\theta) + n^{-\Omega(1)},$$
and by~\eqref{def:Delta20},~\eqref{nus0def} and~\eqref{eq:Delta2:sumnus} we have 
$$\Delta_2 = \nu_2 + O^*(n^{\theta}) + n^{-\Omega(1)} \qquad \text{and} \qquad \Delta_2^0 = \sum_{k = 2}^{r-1} \nu_k^0 = \nu_2^0 + n^{-\Omega(1)}.$$
By~\eqref{plausible:edges}, and observing that $n^{1-\delta} \cdot \Delta_2 / \mu_r = O^*( n^{\theta}) + n^{-\Omega(1)}$ for every $\delta > 0$, since $\Delta_2 = O^*( n^{2\theta}+  n^{\theta})$, it follows that
\begin{equation}\label{eq:eHmu:nunu:calc}
\bigg( \frac{2e(H)}{\mu_r} - 1 \bigg) \big( \Delta_2 - \Delta_2^0 \big) = \nu_2 - \nu_2^0 + O^*( n^{\theta}) + n^{-\Omega(1)},
\end{equation}
and hence we obtain the claimed bound. Finally, recall from Lemma~\ref{lem:plausible:whp} that $H_r(G(n,p))$ is plausible with high probability.
\end{proof}

\subsection{Lower bound for $p$ constant}\label{sslbg}

In this subsection we will complete the proof of Theorem~\ref{thgen:intro}. To do so we will use Theorem~\ref{th2way}, as in Section~\ref{sec:constant:r:proof}, but here we will need to work harder to bound the term $C(H)$ (the `complex' error term), since if $k = {r \choose 2}$ (say) then $\Delta_k$ is very large (to see this, consider sets of $k$ copies of $K_r$ all within a single copy of $K_{2r}$), so the bound given by Lemma~\ref{CYbd1} is too weak. However, such dense configurations simply do not appear in $H_r(G(n,p))$, and we will therefore be able to carry out the strategy outlined at the end of Section~\ref{seclower}.

Our first task is to define the family $\cG$ of hypergraphs for which we will prove the claimed lower bound; doing so will require a little preparation. The first step is to describe the set of configurations of copies of $K_r$ that typically appear in $G(n,p)$ when $\mu_r = n^{3/2 - \eps}$.


\begin{definition}\label{defL}
A cluster is \emph{legal} if it is either contained in a copy of $K_{r+1}$, consists of a pair of copies of $K_r$ sharing exactly $2$ vertices, or is a $3$-cluster formed by cliques whose vertex sets $S_1,S_2,S_3$ satisfy
$$|S_1 \cap S_2| = |S_1 \cap S_3| = 2 \qquad \text{and} \qquad \big( S_2 \cap S_3 \big) \setminus S_1 = \emptyset.$$
We say that a graph $G$, or the corresponding $r$-uniform hypergraph $H$, is \emph{legal} if it contains only legal clusters.
\end{definition}

We will show below that, if $p$ and $r$ are chosen as in Theorem~\ref{thgen:intro}, then $G(n,p)$ is legal with high probability (see Lemma~\ref{Llegal}). 

Next, recall from Definition~\ref{defsc} that a star-cluster is a collection of at least three copies of $K_r$, one of which (the `centre') is not contained in the union of others (the `leaves'), and if there are at least three leaves, then each leaf intersects the centre in some edge that is not contained in any of the other leaves. Here we will need the following slightly more restrictive notion. 

\begin{definition}\label{lsc}
We say that a star-cluster $\cS$ is \emph{legal} if the graph formed by the union of its leaves is legal.
\end{definition}

As in Section~\ref{seclower}, we say that the star-cluster $\cS$ is \emph{pre-present} in $H$ if all of its leaves correspond to edges of $H$, and \emph{present} if the centre also corresponds to an edge of $H$. Now, define 
\begin{equation}\label{def:hatCLH}
\hCL(H) = \sum_{\cS\text{ legal}} \1\big[ \cS \text{ is pre-present in } H \big] \cdot \pi_c(\cS),
\end{equation}
where $\pi_c(\cS)$ denotes the probability that $\cS$ is present in $H_r(G(n,p))$, conditioned on the event that $\cS$ is pre-present in $H_r(G(n,p))$. The following variant of Observation~\ref{obs:CChat} follows from the same proof. 

\begin{obs}\label{obs:CCL}
$C(H) \le \hCL(H)$ for every possible and legal hypergraph $H$. 
\end{obs}

\begin{proof}
To deduce the claimed inequality from the proof of Observation~\ref{obs:CChat}, note that if $H$ is legal, then any star-cluster that is pre-present in $H$ is also legal. 
\end{proof}

We can now define the family of hypergraphs $\cG$ for which we will prove the lower bound in Theorem~\ref{thgen:intro}. 

\begin{definition}\label{defHgood}
Let $C > 0$ be an arbitrary constant. An $r$-uniform hypergraph $H$ with $n$ vertices is \emph{reasonable} if $H$ is plausible, $H$ is legal, and 
$$\hCL(H) \le (\log n)^C \, \Ex\big[ \hCL\big( H_r( G(n,p) ) \big) \big].$$ 
\end{definition}

Theorem~\ref{thgen:intro} is an immediate consequence of the following theorem. 

\begin{theorem}\label{th2way2}
Let $\eps > 0$ be constant, let $-1 + \eps \le \theta = \theta(n) \le 1/2 - \eps$, and suppose that $p = p(n) \in (\eps,1-\eps)$ and $r = r(n) \sim 2\log_{1/p}(n)$ are chosen so that~$\mu_r = \Theta(n^{1+\theta})$. Then 
$$\Pr\big( H_r(G(n,p)) = H \big) = \pi^{e(H)} (1-\pi)^{{n \choose r} - e(H)} p^{-t(H)} e^{-\nu_2+\nu_2^0 + O^*(n^\theta) + n^{-\Omega(1)}}$$
for every reasonable $H$, and $H_r(G(n,p))$ is reasonable with high probability.
\end{theorem}

We will deduce Theorem~\ref{th2way2} from Theorem~\ref{th2way}. Let us assume that $\eps > 0$ is fixed, and that $p \in (\eps,1-\eps)$, $r \in \N$ and $\theta \in \R$ satisfy the assumptions of the theorem, and recall from the previous section that the conditions of Theorem~\ref{thmain} are satisfied. Note also that $1 - p \ge \eps$ by assumption, and that \begin{equation}\label{eq:th2way:conditions:check:pconstant}
\Delta_2^0 = O^*\bigg( \frac{\mu_r^2}{n^2} \bigg) = o(\mu_r),
\end{equation}
since $\mu_r = O(n^{3/2})$. We may therefore apply Theorem~\ref{th2way} to the ${r \choose 2}$-uniform hypergraph encoding copies of $K_r$ in $E(K_n)$. Since, by Definition~\ref{defHgood}, every reasonable hypergraph $H$ is plausible, it follows from the proof of Theorem~\ref{thgen} (in particular from Lemma~\ref{lem:pconstant:error:bound} and~\eqref{eq:eHmu:nunu:calc}) that the claimed bound holds up to a factor of $\exp\big( O(C(H)) \big)$. By Observation~\ref{obs:CCL} and Definition~\ref{defHgood}, it will therefore suffice to prove the following lemma. Recall that $\hb = \max\{\theta,0\}$.


\begin{lemma}\label{lEX'}
$$\Ex\big[ \hCL\big( H_r( G(n,p) ) \big) \big] = O^*\big( n^\theta + n^{3\hb-1} \big).$$
\end{lemma}

We say that a star-cluster $\cS$ is \emph{special} if it consists of exactly three cliques, all of which are contained in a copy of $K_{r+1}$, and otherwise we say that $\cS$ is \emph{ordinary}. The first term in Lemma~\ref{lEX'} comes from special legal star-clusters, and the second from ordinary ones. Note that in an ordinary legal star cluster, no two of the leaves can be contained in a copy of $K_{r+1}$ that also contains the centre of $\cS$, since otherwise these two leaves would already cover all but one edge of the centre, so by the definition of a star-cluster there can be no further leaves in $\cS$, and hence $\cS$ is special.

\begin{proof}[Proof of Lemma~\ref{lEX'}]
Observe first that
\[
 \Ex\big[ \hCL\big( H_r(G(n,p)) \big) \big] = \sum_{\cS\text{ legal}} \Pr\big( \cS\text{ is present in } G(n,p) \big).
\]
The contribution to the sum from special legal star-clusters is 
$$O^*(\mu_{r+1}) = O^*(n^\theta),$$ 
by~\eqref{eq:murplusone}, and since the union of the cliques in $\cS$ is a copy of $K_{r+1}$, and there are $O^*(1)$ ways to choose the centre and the leaves inside a given copy of $K_{r+1}$. From now on we consider only ordinary legal star-clusters.

As usual, we will partition the sum according to the isomorphism type of the ordinary legal star-cluster $\cS$. For a fixed type, we partition the leaves of $\cS$ into groups: any set of leaves contained in a copy of $K_{r+1}$ forms a group, and every other leaf forms a group of size $1$. Note that each group consists of at most three cliques, since no leaf can be contained in the union of two or more other leaves, by the minimality of $\cS$ (see Definition~\ref{defsc}). Observe also that any group shares at most two edges with the union of all the other groups, since the non-singleton groups are edge-disjoint from the others, and each singleton group shares an edge (and no more) with at most two other singleton groups.

We now bound the expectation in the usual way, starting by choosing the centre, $S$, and then adding leaves one by one. Each time we add a singleton group, it forms a $2$-cluster with $S$, and conditioning on the other leaves being present only increases the probability by a factor of at most $p^{-2} = O(1)$, so we multiply the expectation by a factor of 
$$O\bigg( \frac{\Delta_2}{\mu_r} \bigg) = O( n^{\hb-1} ),$$
where the final inequality follows from~\eqref{eq:Delta2:sumnus}. Adding a non-singleton group is slightly more complicated; we claim that each time we add such a group, we multiply the expectation by a factor of 
\begin{equation}\label{eq:non:singleton}
O^*\bigg( \frac{\Delta_2}{\mu_r} \cdot np^r \bigg) = O^*( n^{\hb-2} ),
\end{equation}
where the final inequality holds by Observation~\ref{obs:p:tothe:r}. To see this, let $U$ be the vertex set of a non-singleton group, and note that (as we observed before the proof) $V(S)$ is not contained in $U$, since $\cS$ is ordinary. It follows that at most $r - 1$ vertices of $U$ are shared with the centre of $\cS$, and hence $U$ has at least two vertices outside $V(S)$. Let $u \in U \setminus V(S)$, and let $T$ be the copy of $K_r$ with vertex set $U \setminus \{u\}$. Note that $T$ is not necessarily a leaf of $\cS$; however, if $\cS$ is present in $G(n,p)$ then at least ${r \choose 2} - 1$ edges of $T$ are contained in $G(n,p)$. Moreover, if $\cS$ is present in $G(n,p)$ then at least $r - 1$ edges between $u$ and $V(T)$ are contained in $G(n,p)$, and none of these edges are contained in $S \cup T$, or in any of the cliques in different groups. Since there are only $O^*(1)$ choices for the leaves of $\cS$ given $U$, we obtain~\eqref{eq:non:singleton}. 

Finally, note that since $\cS$ has at least two leaves, we either have at least two singleton groups, or at least one non-singleton group. It follows that
\begin{equation}
 \sum_{\substack{\cS\text{ legal}\\ \text{and ordinary}}} \Pr\big( \cS\text{ is present in } G(n,p) \big) = O^*\big( \mu_r \cdot n^{2\hb-2} \big) = O^*(n^{3\hb-1}),
\end{equation}
as required.
\end{proof}

It only remains to show that $H_r(G(n,p))$ is reasonable with high probability. We have already shown that $H_r(G(n,p))$ is plausible with high probability, and the final condition of reasonableness in Definition~\ref{defHgood} holds by Markov's inequality. We therefore only need to prove the following lemma. 

\begin{lemma}\label{Llegal}
$G(n,p)$ is legal with high probability.
\end{lemma}

\begin{proof}
Recall first from the proof of Lemma~\ref{lem:plausible:whp} that 
$$\nu_3 + \cdots + \nu_{r-2} = o(1),$$
so with high probability the only pairs of copies of $K_r$ in $G(n,p)$ that share at least one edge share either $2$ or $r - 1$ vertices. Let $\cC$ be a cluster, and suppose that $\cC$ contains two copies of $K_r$ sharing $r-1$ vertices, say $E_1$ and $E_2$. Either $\cC$ is contained in a copy of $K_{r+1}$, or it contains a third copy of $K_r$ that shares between $2$ and $r - 1$ vertices with $E_1 \cup E_2$. The expected number of these latter configurations is at most
\[
O^*\bigg(  \binom{n}{r+1} p^{\binom{r+1}{2}} \cdot \frac{\Delta_2}{\mu_r} \bigg) = O^*\bigg( \frac{\Delta_2}{n} \bigg) = o(1),
\]
by Observation~\ref{obs:p:tothe:r} and~\eqref{eq:Delta2:sumnus}, so with high probability no such clusters exist.

It remains to consider clusters $\cC$ in which every pair of copies of $K_r$ share at most two vertices. If $\cC$ has size 2 then it is legal, and otherwise there exist three cliques in $\cC$ whose vertex sets $S_1,S_2,S_3$ satisfy
$$|S_1 \cap S_2| = |S_1 \cap S_3| = 2.$$
The expected number of such configurations with $\big( S_2 \cap S_3 \big) \setminus S_1 \ne \emptyset$ is 
$$O^*\bigg( \frac{\mu_r^3}{n^5} \bigg) = O^*(n^{3\theta - 2}) = o(1),$$ 
and the expected number of $4$-clusters with pairwise intersections of size at most~$2$ is
$$O^*\bigg( \frac{\mu_r^4}{n^6} \bigg) = O^*(n^{4\theta - 2}) = o(1),$$ 
so with high probability these also do not occur in $G(n,p)$, as required.
\end{proof}

Finally, let's put the pieces together and prove Theorem~\ref{th2way2}. 

\begin{proof}[Proof of Theorem~\ref{th2way2}]
The proof of Theorem~\ref{thgen}, 
together with the assumption that $1 - p \ge \eps$ and~\eqref{eq:th2way:conditions:check:pconstant}, which allows us to apply Theorem~\ref{th2way} in place of Theorem~\ref{thmain}, and the fact that every reasonable $H$ is plausible, implies that the claimed bound holds up to a factor of $\exp(O(C(H))$. 

To bound this extra error term, we use Observation~\ref{obs:CCL} and Lemma~\ref{lEX'}. Indeed, since $H$ is possible and legal, we have $C(H) \le \hCL(H)$, and hence 
$$C(H) \le \hCL(H) = O^*\Big( \Ex\big[ \hCL\big( H_r( G(n,p) ) \big) \big] \Big) = O^*\big( n^\theta + n^{-\Omega(1)} \big)$$ 
by Definition~\ref{defHgood} and Lemma~\ref{lEX'}, so we obtain the claimed two-way bound for every reasonable $H$. Finally, observe that $H_r(G(n,p))$ is reasonable with high probability by Lemmas~\ref{lem:plausible:whp} and~\ref{Llegal}, and Markov's inequality.
\end{proof}

\section{Clique factors in $G(n,p)$}\label{secappl}

One of the main motivations for comparing the hypergraph $H_r(G(n,p))$ with the binomial random hypergraph $H_r(n,\pi)$ is to study $K_r$-factors in $G(n,p)$, by relating them to matchings in $H_r(n,\pi)$, or in the model $H_r(n,m)$ where the number of hyperedges is fixed. The main focus of research has been on the threshold for such a factor/matching to exist, for $r \ge 3$ constant. While the original (extremely influential) paper of Johannson, Kahn and Vu~\cite{JKV} established the (coarse) threshold simultaneously in both models, the more recent extremely precise results of Kahn~\cite{Kahn-asymp,Kahn-hitting} were only proved directly for the hypergraph matching question (`Shamir's problem'); corresponding results for $K_r$-factors are deduced using the coupling of \cite{R-copies} (for the sharp threshold for $r\ge 4$) and the extension due to Heckel~\cite{Heckel-copies} for $r=3$. Remarkably, despite the apparent loss in the coupling,
Heckel, Kaufmann, M\"uller and Pasch~\cite{HKMP-hitting} managed to transfer even Kahn's hitting time result~\cite{Kahn-hitting} to the $K_r$-factor setting.

For both Shamir's problem and the $K_r$-factor problem, the threshold for existence is now extremely well understood. Here we turn to a question that has received less attention: once hypergraph matchings/clique factors exist, how~many are there? For simplicity, we will focus on the case where $p=p(n)$ or $m=m(n)$ is a (possibly small) constant factor above the threshold.

Our main aim in this subsection is to prove Theorem~\ref{thm:factorpweak}; however, we will begin by discussing matchings in random hypergraphs. Let $M(H)$ denote the number of perfect matchings in an $r$-uniform hypergraph $H$, and recall from Section~\ref{sec:intro:factors} that we write
\[
\Sigma(n,m) = \frac{n!}{r!^{n/r}(n/r)!} \frac{(m)_{n/r}}{(N)_{n/r}}
\]
for the expected number of such matchings in the random hypergraph $H_r(n,m)$ if $r \,|\, n$, where as usual $N = \binom{n}{r}$ is the number of possible hyperedges and $(m)_k = m(m-1)\cdots (m-k+1)$. Let us also write
\[
\Sigma(n,\pi) = \frac{n!}{r!^{n/r}(n/r)!} \pi^{n/r},
\]
for the corresponding expectation in the binomial random hypergraph $H_r(n,\pi)$, and let
\[
 m_r(n) = \frac{n \log n}{r} \qquad \text{and} \qquad \pi_r(n) = \frac{m_r(n)}{N}
\]
be the thresholds for a perfect matching to exist in the two models when $r \in \N$ is constant (note also that $\pi_r(n) \sim q_r(n)^{r \choose 2}$, where $q_r(n)$ was defined in~\eqref{p0def}). Observe that if $\pi = m/N$, then 
\begin{equation}\label{ratio}
 \frac{\Sigma(n,\pi)}{\Sigma(n,m)} = \frac{m^{n/r}}{N^{n/r}}\frac{(N)_{n/r}}{(m)_{n/r}}
 = \exp\left(\frac{n^2}{2r^2m}+O(n^3/m^2)+O(n/m)\right),
\end{equation}
which is $\exp(\Theta(n/\log n))$ when $m = \Theta(m_r(n))$ and $r = O(1)$.

Kahn~\cite[Theorem 1.2]{Kahn-asymp} gives a counting result for Shamir's problem, which can be strengthened to the following; see Section~\ref{ssK+}.

\begin{theorem}\label{thhm}
Let $r \ge 3$ and\/ $\eps > 0$ be fixed. If\/ $m \ge (1+\eps)m_r(n)$, then  
\[
 M(H_r(n,m)) = e^{o(n/\log n)} \, \Sigma(n,m)
\]
with high probability.
\end{theorem}

Only the lower bound here is interesting; the upper bound is just Markov's inequality. An immediate corollary is the corresponding result for $H_r(n,\pi)$.

\begin{corollary}\label{cor:chm}
Let $r \ge 3$ and\/ $\eps > 0$ be fixed. If\/ $\pi \ge (1+\eps)\pi_r(n)$, then 
\[
 M(H_r(n,\pi)) = e^{o(n/\log n)} \, \Sigma(n,m)
\]
with high probability, where $m = \big\lfloor \binom{n}{r} \pi \big\rfloor$.
\end{corollary}

\begin{proof}
A standard argument: if we take $m_+ = (1+n^{-1/2}) m$ and $m_- = (1-n^{-1/2})m$, then $H_r(n,\pi)$ with high probability has between $m_-$ and $m_+$ edges, and conditional on having $m$ edges, it has the distribution of $H_r(n,m)$. The result now follows from Theorem~\ref{thhm}, since the dependence of $\Sigma(n,m)$ on $m$ scales (roughly) as $m^{O(n)}$, so $\Sigma(n,m_+)$ and $\Sigma(n,m_-)$ are both within a factor of $\exp(O(n^{1/2}))$ of $\Sigma(n,m)$.
\end{proof}

As we see here, translating between $H_r(n,m)$ and $H_r(n,\pi)$ is essentially trivial in either direction, since the property of containing a perfect matching is monotone. Note, however, that the statement for $H_r(n,\pi)$ involves the expectation in the model $H_r(n,m)$. By~\eqref{ratio}, the expected number of perfect matchings in $H_r(n,\pi)$ is larger by a factor of order $\exp\big( \Theta(n^2/m) \big)$, and this difference cannot be absorbed into the error term when $m = \Theta(n\log n)$.

This was for counting matchings in hypergraphs. What can we say about the number of $K_r$-factors in $G(n,m)$ or $G(n,p)$? As outlined above, the coupling results of~\cite{R-copies} and Heckel~\cite{Heckel-copies} give the following corollary of Theorem~\ref{thhm}. Here we need some upper bound on $p$ for the coupling to work well; presumably it is not needed for the result to be true. Recall that we write $F_r(G)$ for the number of $K_r$-factors in a graph $G$.  

\begin{corollary}\label{cor:factorl}
Let $\eps > 0$ and $r \ge 3$ be fixed. There exists $\gamma = \gamma(r) > 0$ such that if $p = p(n)$ satisfies $(1+\eps)q_r(n) \le p \le n^{-2/r + \gamma}$, then
\[
 F_r(G(n,p)) \ge e^{o(n/\log n)} \, \Sigma(n,m)
\]
with high probability, where $m = \big\lfloor \binom{n}{r}p^{\binom{r}{2}} \big\rfloor$.
\end{corollary}

\begin{proof}
This is essentially immediate from Corollary~\ref{cor:chm} and~\cite[Theorem 1]{R-copies} for $r \ge 4$, or~\cite[Theorem~2]{Heckel-copies} for $r = 3$: the latter results show that there is a coupling of $G(n,p)$ and $H_r(n,\pi')$ with $\pi'\sim p^{\binom{r}{2}}$ such that with high probability the set of copies of $K_r$ in the former contains all hyperedges of the latter. As noted in both papers (see Remark~4 in~\cite{R-copies}), we may in fact take $\pi' = (1-n^{-\delta})\pi$ for some constant $\delta > 0$. Then the corresponding $m$ and $m'$ agree within a factor of $1 - n^{-\delta}$, so $\Sigma(n,m') = \exp(O(n^{1-\delta})) \cdot \Sigma(n,m).$
\end{proof}

The discussion above suggests an obvious question: can we bound the number of $K_r$-factors in $G(n,p)$ from above? Of course, the expectation gives an upper bound. However, writing $\pi = p^{\binom{r}{2}}$, we have 
\begin{equation}\label{EFr}
 \Ex\big[ F_r(G(n,p)) \big] = \Sigma(n,\pi),
\end{equation}
which by \eqref{ratio} is significantly larger than $\Sigma(n,m)$ for $m = \floor{\mu_r}$, and so does not match the lower bound in Corollary~\ref{cor:factorl}. One might expect that switching to $G(n,m)$ would help, but this is not the case. In the hypergraph setting, the expectation in $H_r(n,\pi)$ is `too large' (larger than the typical value) because when $H_r(n,\pi)$ has more hyperedges than typical, it has a lot more matchings. However, the expectation in $G(n,p)$ is not too large due to extra edges, but rather due to extra copies of $K_r$. 

This suggests a strategy: condition on the total number of copies of $K_r$ in $G(n,p)$. This thinking (in the different context of studying colourings of $G(n,1/2)$, which we will return to in future work) is in fact what motivated the present paper; using Theorem~\ref{thm:Kr} we are able to prove the following result.

\begin{theorem}\label{thm:factorp}
Let $r \ge 3$ and $\eps > 0$ be constants. There exists $\gamma = \gamma(r) > 0$ such that if\/ $p \le n^{-2/r + \gamma}$, then
\[
 F_r(G(n,p)) \le e^{O^*(\sqrt{n})} \, \Sigma(n,m)
\]
with high probability, where $m = \big\lfloor \binom{n}{r}p^{\binom{r}{2}} \big\rfloor$.
\end{theorem}

Together, with Corollary~\ref{cor:chm}, this result establishes Theorem~\ref{thm:factorpweak}, illustrating how our present results here complement the couplings of~\cite{R-copies} and~\cite{Heckel-copies}. Note that the result above is (of course) only interesting when $p$ is at least (or is close to) the threshold $q_r(n)$, since otherwise $F_r(G(n,p)) = 0$ with high probability.

\subsection{A general lemma}

The main ingredient in the proof of Theorem~\ref{thm:factorp} is a general lemma showing how to calculate using the probability distribution of $H_r(G(n,p))$ given by Theorem~\ref{thm:Kr:intro}. In order state it, we will need one more definition. 

Throughout this subsection, let us fix $r \ge 3$ and $p = p(n) \le n^{-2/r+\gamma}$, where $\gamma > 0$ is sufficiently small. Recall from Definition~\ref{defKrgood} the definition of a good hypergraph, and let $W_3(H)$ denote the number of 3-clusters in $H$, so 
$$\Delta_3 = \Ex\big[ W_3\big( H_r( G(n,p) ) \big) \big],$$
cf.~\eqref{Dkdef}. Recall also that $\omega = \omega(n) \to \infty$ as $n \to \infty$ arbitrarily slowly. 

\begin{definition}\label{def_wb}
An $r$-uniform hypergraph $H$ is \emph{well behaved} if it is good and
$$W_3(H) \le \omega \Delta_3.$$
\end{definition}

Recall from Lemma~\ref{lem:Hgood:whp} that $H_r(G(n,p))$ is good with high probability, and observe that therefore, by Markov's inequality, $H_r(G(n,p))$ is well behaved with high probability. Let $\cW$ denote the event that $H_r(G(n,p))$ is well behaved, and for each $m \in \N$ with $|m-\mu_r|\le \omega\sqrt{\mu_r}$, consider the event
\[
 \cW_m = \big\{ H_r(G(n,p)) \text{ is well behaved and } e(H_r(G(n,p))) = m \big\}.
\]
Recall also that $\xi$ was defined in~\eqref{xidef}. 

The following technical lemma is the main result of this subsection. 

\begin{lemma}\label{lmoment}
For each $r \ge 3$, there exists $\gamma > 0$ such that if $p \le n^{-2/r+\gamma}$ and $\mu_r \to \infty$, then the following holds. Let $E_1,\ldots,E_k$ be distinct copies of $K_r$ in $K_n$, and let $F = E_1 \cup \cdots \cup E_k$. For any $m \in \N$ with $|m - \mu_r|\le \omega\sqrt{\mu_r}$, we have
\[ 
\frac{\Pr\big( \cW_m \cap \{ F \subset G(n,p)\} \big)}{\Pr\big( \Bin(N,\pi)=m \big)} \le \frac{(m)_k}{(N)_k} \cdot p^{-t(F)} \cdot e^{- \binom{k}{2}\binom{m}{2}^{-1} ( \Delta_2 \,-\, \Delta_2^0 ) \,+\, O^*(\xi) \,+\, n^{-\Omega(1)} },
\]
where $N = \binom{n}{r}$ and
\begin{equation}\label{tF:def}
 t(F) = k \binom{r}{2} - e(F)
\end{equation}
is the number of repeated graph edges in $E_1,\ldots,E_k$. 
\end{lemma}

We remark that the error terms in the lemma depend on $r$, $\gamma$ and $\omega$, but not on $k$. To help motivate the bound in the lemma, notice that
$$\frac{(m)_k}{(N)_k} \cdot \Pr\big( \Bin(N,\pi) = m \big)$$
is the probability that $H_r(n,\pi)$ has exactly $m$ edges, including $E_1,\ldots,E_k$. 

Recall from Section~\ref{ssGnp} that if $H$ is a hypergraph, then we write $t_s(H)$ for the number of pairs of hyperedges of $H$ sharing exactly $s$ vertices, and from~\eqref{tdef} that $t(H)$ is the analogue of~\eqref{tF:def} for hypergraphs. In the proof of Lemma~\ref{lmoment}, we will need the following bounds on $t(H)$ in terms of $t_2(H),\ldots,t_{r-1}(H)$. 


\begin{lemma}\label{ltH}
For any $r$-uniform hypergraph $H$ we have
\begin{equation*} 
 0 \le  \sum_{s=2}^{r-1} \binom{s}{2} t_s(H) - t(H) \le  \binom{r-1}{2}W_3(H).
\end{equation*}
\end{lemma}

\begin{proof}
Both $\sum_{s=2}^{r-1} \binom{s}{2}t_s(H)$ and $t(H)$ count graph edges that appear in at least two hyperedges of $H$, with certain multiplicities. Edges that appear in exactly $k$ hyperedges have multiplicities $\binom{k}{2}$ and $k-1$, which are equal when $k = 2$, and for $k \ge 3$ the former multiplicity is larger, so this proves the lower bound. 

For the upper bound, note that each edge in $k$ hyperedges contributes zero to the central term if $k = 2$, and at most ${k \choose 2} - k + 1 \le {k \choose 3}$ if $k \ge 3$. Moreover, each such edge is contained in the common intersection of at least ${k \choose 3}$ different $3$-clusters. The central term is therefore at most the sum over all $3$-clusters $\cC$ of the number of graph edges contained in the common intersection of the hyperedges in $\cC$, which is at most the claimed upper bound. 
\end{proof}

We will also need a similar bound for the following variant of $t_s(H)$.

\begin{definition}\label{deftm}
Given an $r$-uniform hypergraph $H$ and $2 \le s \le r-1$, we write $t_s^-(H)$ for the number of unordered pairs $\{e,f\}$ of hyperedges of $H$ such that
\begin{itemize}
\item[(i)] $e$ and $f$ share exactly $s$ vertices, and 
\item[(ii)] neither $e$ nor $f$ intersects any other hyperedge of $H$ in two or more vertices.
\end{itemize}
\end{definition}

In other words, $t_s^-(H)$ counts the number of `isolated' $2$-clusters with intersection size $s$. In particular, note that $t_s^-(H) \le t_s(H)$.

\begin{lemma}\label{ttm}
For any $r$-uniform hypergraph $H$ we have
\begin{equation*} 
 0 \le \sum_{s=2}^{r-1} \binom{s}{2}t_s(H) - \sum_{s=2}^{r-1} \binom{s}{2}t^-_s(H) \le    3\binom{r-1}{2}W_3(H).
\end{equation*}
\end{lemma}

\begin{proof}
The lower bound is trivial, since $t_s^-(H) \le t_s(H)$ by definition. Now, observe that $\sum_{s=2}^{r-1} \binom{s}{2}t_s(H)$ counts all pairs $(e,\cC)$, where $\cC$ is a $2$-cluster and $e$ is a graph edge in the intersection of the two hyperedges of $\cC$. If the pair $(e,\cC)$ does not contribute to the sum with $t_s^-$, then there must be some hyperedge of $H$ that forms a 3-cluster with $\cC$. On the other hand, for each 3-cluster in $H$ there are at most three choices for the $2$-cluster $\cC$, and then at most $\binom{r-1}{2}$ choices for the edge $e$ in the intersection of these two hyperedges. 
\end{proof}

We are now ready to prove our technical lemma. 

\begin{proof}[Proof of Lemma~\ref{lmoment}]
Let $\cF_m$ be the set of $r$-uniform hypergraphs on $[n]$ that contain the $k$ hyperedges corresponding to $E_1,\ldots,E_k$ and satisfy the conditions of $\cW_m$, and observe that
\begin{equation}\label{eq:tech:1}
\Pr\big( \cW_m \cap \{ F \subset G(n,p) \} \big) = \sum_{H \in \cF_m} \Pr\big( H_r(G(n,p)) = H \big).
\end{equation}
Every $H \in \cF_m$ is good by Definition~\ref{def_wb}, so by Theorem~\ref{thm:Kr}, the right-hand side of~\eqref{eq:tech:1} is equal to 
\begin{equation}\label{Pform}
\sum_{H \in \cF_m} \pi^{m} (1-\pi)^{N-m} p^{-t(H)} \exp\bb{-\Delta_2+\Delta_2^0+O^*(\xi)+n^{-\Omega(1)}}. 
\end{equation}
Note that $p^{-t(H)}$ is the only term that depends on $H$, so our aim is to bound
$$\sum_{H \in \cF_m} p^{-t(H)}.$$
Roughly speaking, the idea is to build a random hypergraph $H$ starting with the $k$ given edges and adding $m-k$ new random edges one-by-one, considering how the $p^{-t(H)}$ term evolves as we add edges, which we can conveniently do working in terms of expectations.

Let $h_1,\ldots,h_k$ be the hyperedges corresponding to $E_1,\ldots,E_k$, and let the remaining edges $h_{k+1},\ldots,h_m$ be chosen independently and uniformly at random from all $N = \binom{n}{r}$ possible hyperedges. For $k \le i \le m$ let $\tH_i$ be the random hypergraph with edge-set $\{h_1,\ldots,h_i\}$. Since $H \in \cF_m$ implies that $H$ has $m$ distinct hyperedges, there are exactly $(m-k)!$ outcomes leading to $\tH_m = H$ for any given $H \in \cF_m$. Thus, in this probability space, 
\[
 \Pr( \tH_m = H ) = \frac{(m-k)!}{N^{m-k}},
\]
and hence
\begin{align}\label{eq:tech:2}
\sum_{H \in \cF_m} p^{-t(H)} & \, = \frac{N^{m-k}}{(m-k)!} \sum_{H \in \cF_m} p^{-t(H)} \cdot \Pr\big( \tH_m = H \big) \nonumber\\
& \, = \frac{N^{m-k}}{(m-k)!} \cdot \Ex\pigl[ p^{-t(\tH_m)} \1\big[ \tH_m \in \cF_m \big] \pigr],
\end{align}
where the expectation is over the random choice of $h_{k+1},\ldots,h_m$. 

To bound the expectation, we will first replace $t(\tH_m)$ by a simpler quantity which is bounded from above and below by the sums in Lemmas~\ref{ltH} and~\ref{ttm}. To do so, for each $2 \le i \le m$ and $2 \le s \le r-1$ let us write $A_{i,s}$ for the event that $h_i$ shares exactly $s$ vertices with one previous $h_j$, and does not meet any other $h_{j'}$ with $j' < i$ in two or more vertices. Now, for each $2 \le \ell \le m$, define
\[
Z_\ell = \sum_{i = 2}^\ell \sum_{s = 2}^{r-1} \1\big[ A_{i,s} \big] \binom{s}{2},
\]
and observe that 
\[
 \sum_{s=2}^{r-1} {s \choose 2} t_s^-(\tH_\ell) \le Z_\ell \le  \sum_{s=2}^{r-1} {s \choose 2} t_s(\tH_\ell),
\]
since $Z_\ell$ counts the graph edges in a subset of the intersections between pairs of hyperedges of $\tH_\ell$ (this proves the upper bound), and this subset includes all isolated $2$-clusters (this proves the lower bound). By Lemmas~\ref{ltH} and~\ref{ttm}, it follows that
$$t(\tH_\ell) = Z_\ell + O(W_3(\tH_\ell)).$$
Applying this for $\ell = k$ and $\ell = m$, and noting that $W_3(\tH_k) \le W_3(\tH_m)$, we obtain
\[
 t(\tH_m) = t(\tH_k) + Z_m - Z_k + O(W_3(\tH_m)).
\]
Now, by Definition~\ref{def_wb} and Lemma~\ref{lcKrclust}, if $H$ is well behaved, then 
$$W_3(H) \le \omega\Delta_3 = O(\omega\xi)=O^*(\xi).$$
Thus, on the event $\tH_m\in \cF_m$, we have
\[
 t(\tH_m) = t(\tH_k) + Z_m - Z_k + O^*(\xi).
\]
Recalling~\eqref{eq:tech:2}, this implies that
\begin{equation}\label{Sbd2}
\sum_{H \in \cF_m} p^{-t(H)} \le e^{O^*(\xi)} \frac{N^{m-k}}{(m-k)!} \cdot p^{-t(F)} \cdot \Ex\pigl[ p^{-Z_m+Z_k} \pigr],
\end{equation}
since $\log(1/p) = O^*(1)$ and $t(\tH_k) = t(F)$, by definition.

To complete the proof, we will inductively bound $\Ex\big[  p^{-Z_\ell+Z_k} \big]$. To be precise, we claim that for each $k < \ell \le m$, we have 
\begin{equation}\label{eq:tech:3}
 \Ex\pigl[ \, p^{-Z_\ell + Z_k} \,\big|\, \tH_{\ell-1} \pigr] \le \big( 1 + c(\ell) \big) \cdot p^{-Z_{\ell-1}+Z_k},
\end{equation}
where
\[
c(\ell) = (\ell - 1) \sum_{s = 2}^{r-1} \binom{r}{s} \binom{n-r}{r-s} {n \choose r}^{-1} \big( p^{-\binom{s}{2}} - 1 \big).
\]
To show this, let us write
\[
q(s) := \binom{r}{s}\binom{n-r}{r-s} {n \choose r}^{-1}
\]
for the probability that a random $r$-set shares exactly $s$ vertices with a given $r$-set, and observe that
\[
 \Pr\big( A_{\ell,s} \mid \tH_{\ell-1} \big) \le (\ell-1)q(s)
\]
for each $2 \le s \le r - 1$, by the union bound over edges of $\tH_{\ell-1}$. It follows that 
$$\Pr\bigg( Z_\ell - Z_{\ell-1} = {s \choose 2} \;\Big|\; \tH_{\ell-1} \bigg) \le (\ell-1)q(s)$$
for each $2 \le s \le r - 1$, and otherwise $Z_\ell = Z_{\ell-1}$. 
This proves~\eqref{eq:tech:3}, and it follows that
\begin{equation}\label{Exbd}
 \Ex\pigl[ p^{-Z_m+Z_k} \pigr] \le \prod_{\ell = k+1}^m \big( 1 + c(\ell) \big) \le \exp\bigg( \sum_{\ell = k+1}^m c(\ell) \bigg).
\end{equation}

It only remains to bound the sum in~\eqref{Exbd}. To do so, observe first that 
\begin{equation}\label{csum}
 \sum_{\ell = k+1}^m c(\ell) = \bigg( \binom{m}{2}-\binom{k}{2} \bigg) \sum_{s = 2}^{r-1} q(s) \big( p^{-\binom{s}{2}} - 1 \big),
\end{equation}
since $\sum_{\ell=k+1}^m (\ell-1) = \binom{m}{2} - \binom{k}{2}$. Note also that
\[
 \nu_s^0 = \frac{q(s)}{2} \binom{n}{r}^2 p^{2\binom{r}{2}} = q(s) \cdot \frac{\mu_r^2}{2},
\]
by~\eqref{nusdef} and~\eqref{nus0def}, and the definition of $q(s)$, and hence 
\[
q(s) \cdot \binom{m}{2} = \big( 1 + O^*(1/\sqrt{\mu_r}) \big) \cdot \nu_s^0,
\]
since $| m - \mu_r| \le \omega\sqrt{\mu_r}$. Since $\nu_s=p^{-\binom{s}{2}}\nu_s^0$, it follows from~\eqref{csum} that
\begin{align}\label{csum2}
\sum_{\ell = k+1}^m c(\ell) & \, = \big( 1 + O^*(1/\sqrt{\mu_r}) \big) \bigg( 1 - \binom{k}{2}\binom{m}{2}^{-1} \bigg) \sum_{s = 2}^{r-1} \big( \nu_s - \nu_s^0 \big) \nonumber\\
& \, = \bigg( 1 - \binom{k}{2}\binom{m}{2}^{-1} \bigg) \big( \Delta_2 - \Delta_2^0 \big) + O^*\bigg( \frac{\Delta_2}{\sqrt{\mu_r}} \bigg),
\end{align}
since $\Delta_2 = \sum_{s = 2}^{r-1} \nu_s$ and $\Delta_2^0 = \sum_{s = 2}^{r-1} \nu_s^0$.

We are finally ready to put the pieces together. First, by~\eqref{eq:tech:1} and~\eqref{Pform}, 
$$\Pr\big( \cW_m \cap \{ F \subset G(n,p) \} \big) \le \pi^{m} (1-\pi)^{N-m} e^{-\Delta_2+\Delta_2^0+O^*(\xi)+n^{-\Omega(1)}} \sum_{H \in \cF_m} p^{-t(H)}.$$
Next, by~\eqref{Sbd2},~\eqref{Exbd} and~\eqref{csum2}, we have 
$$\sum_{H \in \cF_m} p^{-t(H)} \le \frac{N^{m-k}}{(m-k)!} \cdot p^{-t(F)} \cdot e^{( 1 - \binom{k}{2}\binom{m}{2}^{-1} ) ( \Delta_2 - \Delta_2^0 ) + O^*(\xi) + O^*(\Delta_2/\sqrt{\mu_r})}.$$
Noting that 
$${N \choose m}^{-1} \frac{N^{m-k}}{(m-k)!} = e^{O(m^2/N)} \frac{(m)_k}{(N)_k} = e^{O(\xi)} \frac{(m)_k}{(N)_k},$$
since $k \le m$ and $m^2/N = O(n^{2-r+O(\gamma)}) = O(\xi)$, and recalling from~\eqref{ratxi} that
$$\frac{\Delta_2}{\sqrt{\mu_r}} = O(\xi) + n^{-\Omega(1)},$$ 
it follows that
$$\frac{\Pr\big( \cW_m \cap \{ F \subset G(n,p)\} \big)}{\Pr\big( \Bin(N,\pi) = m \big)} \le \frac{(m)_k}{(N)_k} \cdot p^{-t(F)} \cdot e^{- \binom{k}{2}\binom{m}{2}^{-1} ( \Delta_2 - \Delta_2^0 ) + O^*(\xi) + n^{-\Omega(1)}},$$
as required. 
\end{proof}

It is straightforward to deduce Theorem~\ref{thm:factorp} from Lemma~\ref{lmoment}. 

\begin{proof}[Proof of Theorem~\ref{thm:factorp}]
We may assume that $p \ge q_r(n)/2$, since otherwise we would have $F_r(G(n,p)) = 0$ with high probability. In particular, it follows from this assumption that $p \ge n^{-2/r}$, and thus $\mu_r \to \infty$ and $\xi = \Omega(1)$. Applying Lemma~\ref{lmoment} to each of the $K_r$-factors in $K_n$, and setting $k = n/r$, we obtain
\[
\frac{\Ex\big[ F_r(G(n,p)) \ind{\cW_m}  \big]}{\Pr\big( \Bin(N,\pi)=m \big)} \le F_r(K_n) \cdot \frac{(m)_k}{(N)_k} \cdot e^{- \binom{k}{2}\binom{m}{2}^{-1} ( \Delta_2 - \Delta_2^0 ) \,+\, O^*(\xi)},
\]
for each $m$ with $|m-\mu_r|\le \omega\sqrt{\mu_r}$, since if $F$ is a $K_r$-factor then $t(F) = 0$. 

Now, since $\Delta_2 \ge \Delta_2^0$, and noting that
$$\Sigma(n,m) = F_r(K_n) \cdot \frac{(m)_{n/r}}{(N)_{n/r}},$$
it follows that
$$\frac{\Ex\big[ F_r(G(n,p)) \ind{\cW_m}  \big]}{\Pr\big( \Bin(N,\pi)=m \big)} \le e^{O^*(\xi)} \,\Sigma(n,m).$$
Summing over all $m$ in the range, it follows that
$$\Ex\big[ F_r(G(n,p)) \ind{\cW} \big] \le e^{O^*(\xi)} \, \Sigma(n,m^+),$$
where $\cW$ is the event that $H_r(G(n,p))$ is well behaved (which in particular implies that $|m-\mu_r|\le \omega\sqrt{\mu_r}$), and $m^+ = m + \omega\sqrt{m}$. 
Since 
$$\frac{\Sigma(n,m^+)}{\Sigma(n,m)} = \bigg( \frac{m^+}{m} \bigg)^{O(n)} = \exp\bigg( O^*\bigg( \frac{n}{\sqrt{m}} \bigg) \bigg) = e^{O^*(\sqrt{n})}$$ 
and $\xi \le n^{1/3 + O(\gamma)}$, it follows that
\[
\Ex\big[ F_r(G(n,p)) \ind{\cW} \big] \le e^{O^*(\sqrt{n})} \, \Sigma(n,m).
\]
Since $\cW$ holds with high probability, by Markov's inequality the right-hand side above is an upper bound on $F_r(G(n,p))$ that holds with high probability, completing the proof.
\end{proof}

Replacing the trivial bound $\Delta_2 \ge \Delta_2^0$ with a simple calculation gives a stronger result for fixed $m$. To be precise, one can easily check that
$$\Delta_2 = \big( 1 + o(1) \big) \frac{\mu_r^2}{p}\binom{r}{2}^2 \binom{n}{2}^{-1}$$ 
and $\Delta_2^0 \sim p \Delta_2 = o(\Delta_2)$, and hence, for $k = n/r$ and $m \sim \mu_r$, we have
$$\binom{k}{2}\binom{m}{2}^{-1} \big( \Delta_2 - \Delta_2^0 \big) \sim \frac{(r-1)^2}{2p}.$$ 
It therefore follows from the proof above that
$$\frac{\Ex\big[ F_r(G(n,p)) \ind{\cW_m} \big]}{\Pr\big( \Bin(N,\pi)=m \big)} \le \exp\bigg( - (1+o(1))\frac{(r-1)^2}{2p} \,+\, O^*(\xi) \bigg) \cdot \Sigma(n,m)$$
for every $m \in \N$ with $|m-\mu_r|\le \omega\sqrt{\mu_r}$. This bound illustrates an interesting phenomenon: in $G(n,p)$, given that there are exactly $m$ copies of $K_r$, the conditional expected number of $K_r$-factors is noticeably (around $\exp(-\Theta(p^{-1}))$) smaller than the expected number of matchings in $H_r(n,m)$. Intuitively, the reason for this is that $H_r(G(n,p))$ prefers intersecting pairs of copies of $K_r$ over disjoint ones, so given that there are $m$ copies of $K_r$, we expect more intersecting pairs 
than in $H_r(n,m)$, and so fewer factors.

\subsection{The lower bound}\label{ssK+}

In this section we outline the proof of Theorem~\ref{thhm}, which as discussed earlier is a slight modification of Kahn's proof of Theorem 1.2 in~\cite{Kahn-asymp}.

\begin{proof}[Proof of Theorem~\ref{thhm}]
As in~\cite{Kahn-asymp}, start with the complete $r$-uniform hypergraph and remove hyperedges one by one uniformly at random, letting $\Phi_t$ be the number of matchings after $t$ steps, so the random variable we are interested in is $\Phi_{N-m}$ where $N=\binom{n}{r}$. Let $\xi_t$ be the random fraction of matchings removed in step $t$, and $\gamma_t$ its expectation, which is simply $(n/r)/(N-t+1)$, since each of the $n/r$ edges in some particular matching has probability $1/(N-t+1)$ of being chosen for removal in step $t$. Note that
\begin{equation}\label{PhiM}
 \Phi_m = \Phi_0 \prod_{t=1}^{N-m} (1-\xi_t),
\end{equation}
and, considering the probability that any particular matching survives,
\[
 \E[\Phi_m] = \Phi_0 \prod_{t=1}^{N-m} (1 - \gamma_t).
\]
One key point (of many!) in~\cite{Kahn-asymp} is that $\alpha_t = \xi_t-\gamma_t$ is a martingale difference sequence, and furthermore that (off some bad event, on which he freezes the martingale) $\xi_t = O(\gamma_t)$ (see~(16) in~\cite{Kahn-asymp}), so $\alpha_t = O(\gamma_t)$. It follows that $\sum_{t=1}^{N-m} \alpha_t = O(\sqrt{S})$ with high probability, where $S = \sum_{t=1}^{N-m}\gamma_t^2=O(n/\log n)$. To evaluate the product in \eqref{PhiM}, one considers the log. Since $\xi_t = O(\gamma_t) = o(1)$, we have $\log(1-\xi_t) = - \xi_t + O(\gamma_t^2)$, and this is where the final error in~\cite{Kahn-asymp} comes from: it is written as $o(n)$, but is in fact $O(S) = O(n/\log n)$.

We can improve on this by considering the variance of the $\xi_t$. Although not essential, it is cleaner to consider the product
\[
 \frac{\Phi_m}{\E[\Phi_m]} = \prod_{t=1}^{N-m} \frac{1-\xi_t}{1-\gamma_t}
 = \prod_{t=1}^{N-m} \left(1- \frac{\xi_t-\gamma_t}{1-\gamma_t}\right).
\]
Let $\alpha_t'=(\xi_t-\gamma_t)/(1-\gamma_t)=\alpha_t/(1-\gamma_t)$. This is again a martingale difference sequence, and each term is $O(\gamma_t)$. So we again have that $\sum_{t=1}^{N-m} \alpha_t' = O(\sqrt{S})$ with high probability. Moreover, since $\alpha_t' = o(1)$ for all $t \in [N - m]$, we have
\[
 \log \frac{\Phi_m}{\E[\Phi_m]} = - \sum_{t=1}^{N-m} \alpha_t' - \frac{1+o(1)}{2} \sum_{t=1}^{N-m} (\alpha_t')^2.
\]
Let $\cF_t$ denote the $\sigma$-algebra corresponding to the information revealed after $t$ steps, and define
\[
 V_t = \Ex\big[ (\alpha_t')^2 \mid \cF_{t-1} \big] = (1-\gamma_t)^{-2}\Var\big( \xi_t \mid \cF_{t-1} \big).
\]
Since $\alpha_t'=O(\gamma_t)$ we have that $(\alpha_t')^2 - V_t$ is a martingale difference sequence with each term of size $O(\gamma_t^2) = o(\gamma_t)$, and it follows that the sum of these terms is with high probability $o(\sqrt{S})$, so it just remains to bound $\sum_{t=1}^{N-m} V_t$. If we only use the bound $V_t = O(\gamma_t^2)$ effectively used in \cite{Kahn-asymp}, we are still left with an $O(n/\log n)$ error term. However, although it is not stated there, \cite{Kahn-asymp} contains all the ingredients to prove a stronger bound. The key is that, in addition to the deterministic bound $\xi_t = O(\gamma_t)$, 
the property $\mathfrak{E}$ defined just below (62) implies that (off the same bad event) $\xi_t \sim \gamma_t$ with conditional probability $1 - o(1)$. Together, these two bounds imply that $\Var\big( \xi_t \mid \cF_{t-1} \big) = o(\gamma_t^2)$, so $V_t = o(\gamma_t^2)$. Hence we have $\sum_{t=1}^{N-m} V_t = o(S)$, giving the claimed improvement.
\end{proof}

Note that, by the same argument, a stronger result would follow from any improvement of this `approximate' flatness of the distribution of $\xi_t$ given $\cF_{t-1}$.

\section*{Acknowledgements}

The research described in this paper was partly carried out during several visits of the second author to IMPA. The authors are grateful to IMPA for providing us with a wonderful working environment.


\begin{thebibliography}{9}


\bibitem{ABCDJMRS} P.~Allen, J.~Böttcher, J.~Corsten, E.~Davies, M.~Jenssen, P.~Morris, B.~Roberts and J.~Skokan, 
A robust Corrádi--Hajnal theorem, 
\emph{Random Struct. Alg.}, {\bf 65} (2024), 61--130. 

\bibitem{AKR} S.~Antoniuk, N.~Kamčev and C.~Reiher, 
Clique factors in randomly perturbed graphs: the transition points, 
arXiv:2410.11003, 2024.

\bibitem{FLM} A.~Ferber, L.~Hardiman and A.~Mond, 
Counting Hamilton cycles in Dirac hypergraphs, 
\emph{Combinatorica}, {\bf 43} (2023), 665--680. 

\bibitem{FKNP} K.~Frankston, J.~Kahn, B.~Narayanan and J.~Park, 
Thresholds versus fractional expectation-thresholds, 
\emph{Ann. Math.}, {\bf 194} (2021), 475--495. 

\bibitem{Heckel} A.~Heckel, 
Non-concentration of the chromatic number of a random graph, 
\emph{J.~Amer.~Math.~Soc.}, \textbf{34} (2021), 245--260.

\bibitem{Heckel-copies} A.~Heckel,
 Random triangles in random graphs, 
 \emph{Random Struct. Alg.}, {\bf 59} (2021), 616--621. 

\bibitem{HKMP-hitting} A.~Heckel, M.~Kaufmann, N.~M\"uller and M.~Pasch,
 The hitting time of clique factors,
 \emph{Random Struct. Alg.}, {\bf 65} (2024), 275--312.

\bibitem{HP} A.~Heckel and K.~Panagiotou,
 Colouring random graphs: Tame colourings,
 arXiv:2306.07253v2, 2023.
 
\bibitem{HR}  A.~Heckel and O.~Riordan,
 How does the chromatic number of a random graph vary?,
 \emph{J. London Math. Soc.} {\bf 108} (2023), 1769--1815.
 
\bibitem{Janson} S.~Janson,                                                                                   
 Poisson approximation for large deviations,                                                                  
 \emph{Random Struct. Alg.}, {\bf 1} (1990), 221--229.                                                 
                                                                                                              
\bibitem{JKV} A. Johansson, J. Kahn and V. Vu,
 Factors in random graphs, 
 \emph{Random Struct. Alg.}, {\bf 33} (2008), 1--28.

\bibitem{JLS} F.~Joos, R.~Lang and N.~Sanhueza-Matamala, 
Robust hamiltonicity, 
arXiv:2312.15262, 2023.

\bibitem{Kahn-asymp} J.~Kahn,
 Asymptotics for Shamir's problem,
 \emph{Adv. Math.}, {\bf 422} (2023), 109019.

\bibitem{Kahn-hitting} J. Kahn,
 Hitting times for Shamir's problem,
 \emph{Trans. Amer. Math. Soc.}, {\bf 375} (2022), 627--668.
 
\bibitem{KKKOP} D.Y.~Kang, T.~Kelly, D.~Kühn, D.~Osthus and V.~Pfenninger, 
Perfect matchings in random sparsifications of Dirac hypergraphs, 
\emph{Combinatorica} (2024), 1--34.
 
\bibitem{KMP} T.~Kelly, A.~Müyesser and A.~Pokrovskiy, 
Optimal spread for spanning subgraphs of Dirac hypergraphs, 
\emph{J.~Combin.~Theory, Ser.~B}, {\bf 169} (2024), 507--541.
 
\bibitem{KSW} M.~Kwan, R.~Safavi and Y.~Wang, 
Counting perfect matchings in Dirac hypergraphs, 
arXiv:2408.09589, 2024.
 
\bibitem{MP} R.~Montgomery and M.~Pavez-Signé, 
Counting spanning subgraphs in dense hypergraphs,  
\emph{Combin. Probab. Computing}, {\bf 33} (2024), 729--741.
 
\bibitem{MNPS-non} F.~Mousset, A.~Noever, K.~Panagiotou and W.~Samotij,
 On the probability of nonexistence in binomial subsets,
 \emph{Ann. Probab.} {\bf 48} (2020), 493--525.
 
\bibitem{PP} J.~Park and H.T.~Pham, 
A proof of the Kahn--Kalai conjecture, 
\emph{J.~Amer. Math.~Soc.}, {\bf 37} (2024), 235--243. 

\bibitem{PSSS} H.T.~Pham, A.~Sah, M.~Sawhney and M.~Simkin, 
A toolkit for robust thresholds, 
arXiv:2210.03064, 2022.
 
\bibitem{R-copies} O. Riordan,
 Random cliques in random graphs and sharp thresholds for $F$-factors,
 \emph{Random Struct. Alg.}, {\bf 61} (2022), 619--637.

\bibitem{RW-Janson} O.~Riordan and L.~Warnke,
 The Janson inequalities for general up-sets,
 \emph{Random Struct. Alg.}, {\bf 46} (2015), 391--395.

\end{thebibliography}
\end{document}